\documentclass[a4paper,12pt]{article}
\usepackage{amsmath}
\usepackage{amssymb}
\usepackage{setspace}
\usepackage{fullpage}
\usepackage{bbm}
\usepackage{amsthm}
\usepackage[latin1]{inputenc}
\usepackage{xcolor}
\usepackage{hyperref}

\newtheorem{theorem}{Theorem}[section]
\newtheorem{lemma}[theorem]{Lemma}

\newtheorem{corollary}[theorem]{Corollary}

\theoremstyle{definition}
\newtheorem{construction}[theorem]{Construction}
\newtheorem{definition}[theorem]{Definition}

\newtheorem{remark}[theorem]{Remark}

\usepackage{pdfsync}
\def\PG{\mathrm{PG}}

\def\PGammaL{\mathrm{P}\Gamma\mathrm{L}}
\def\PGL{\mathrm{PGL}}

\def\A{\mathcal{A}} \def\B{\mathcal{B}} \def\C{\mathcal{C}}
\def\D{\mathcal{D}}  
 \def\K{\mathcal{K}}
\def\L{\mathcal{L}}

 \def\S{\mathcal{S}}
\def\T{\mathrm{Tr}}

\def\F{\mathbb{F}}

\def\a{\alpha}

\newcommand{\comments}[1]{}

\title{A linear set view on KM-arcs}
\date{}
\author{Maarten De Boeck \thanks{This author is supported by the BOF-UGent (Special Research Fund of Ghent University).\newline UGent, Department of Mathematics, Krijgslaan 281 -- S22, 9000 Gent, Flanders, Belgium.\newline Email: mdeboeck@cage.ugent.be} \and Geertrui Van de Voorde \thanks{This author is a postdoctoral fellow of the Research Foundation Flanders (FWO -- Vlaanderen).\newline UGent, Department of Mathematics, Krijgslaan 281 -- S22, 9000 Gent, Flanders, Belgium.\newline Email: gvdvoorde@cage.ugent.be}}
\begin{document}
\maketitle
\begin{abstract} In this paper, we study KM-arcs of type $t$, i.e. point sets of size $q+t$ in $\PG(2,q)$ such that every line contains $0$, $2$ or $t$ of its points. We use field reduction to give a different point of view on the class of {\em translation} arcs. Starting from a particular $\F_2$-linear set, called an {\em $i$-club}, we reconstruct the {\em projective triads}, the {\em translation hyperovals} as well as the translation arcs constructed by Korchm\'aros-Mazzocca, G\'acs-Weiner and Limbupasiriporn. We show the KM-arcs of type $q/4$ recently constructed by Vandendriessche are translation arcs and fit in this family.

Finally, we construct a family of KM-arcs of type $q/4$. We show that this family, apart from new examples that are not translation KM-arcs, contains all translation KM-arcs of type $q/4$. \end{abstract}

\section{Introduction}

\subsection{KM-arcs}
Point sets in $\PG(2,q)$, the Desarguesian projective plane of over the finite field $\F_q$ of order $q$, that have few different intersections sizes with lines have been a research subject throughout the last decades. A point set $\S$ of \emph{type} $(i_{1},\dots,i_{m})$ in $\PG(2,q)$ is a point set such that for every line in $\PG(2,q)$ the intersection size $\ell\cap\S$ equals $i_{j}$ for some $j$ and such that each value $i_{j}$ occurs as intersection size for some line. In \cite{mig} point sets of type $(0,2,q/2)$ of size $\frac{3q}{2}$ were studied. This led to the following generalisation by Korchmáros and Mazzocca in \cite{km}. 

\begin{definition}
  A \emph{KM-arc of type $t$} in $\PG(2,q)$ is a point set of type $(0,2,t)$ with size $q+t$. A line containing $i$ of its points is called an $i$-secant.
\end{definition}

Originally these KM-arcs were denoted as $(q+t)$-arcs of type $(0,2,t)$ \cite{km} or `$(q+t,t)$-arcs of type $(0,2,t)$' \cite{gw} but in honour of Korchmáros and Mazzocca we denote them by KM-arcs. The following results were obtained in \cite[Theorem 2.5]{gw} and \cite[Proposition 2.1]{km}.

\begin{theorem}
  If $\mathcal{A}$ is a KM-arc of type $t$ in $\PG(2,q)$, $2<t<q$, then
  \begin{itemize}
    \item $q$ is even;
    \item $t$ is a divisor of $q$;
    \item there are $\frac{q}{t}+1$ different $t$-secants to $\mathcal{A}$, and they are concurrent.
  \end{itemize}
\end{theorem}

If $\A$ is a KM-arc of type $t$, then the point contained in all $t$-secants to $\A$ is called the \emph{$t$-nucleus} of $\A$.

\begin{definition}
  A point set $\S$ in $\PG(2,q)$ is a called a \emph{translation set} with respect to the line $\ell$ if the group of elations with axis $\ell$ fixing $\S$ acts transitively on the points of $\S\setminus\ell$; the line $\ell$ is called the \emph{translation line}. If a KM-arc is a translation set, then it is called a {\em translation KM-arc}.
\end{definition}

\begin{theorem}[{\cite[Proposition 6.2]{km}}]
  If $\S\subset\PG(2,q)$ is a translation KM-arc of type $t$ with respect to the line $\ell$, then $\ell$ is a $t$-secant to $\S$.
\end{theorem}

The main questions in the study of the KM-arcs are the following: \emph{for which values of $q$ and $t$ does a KM-arc of type $t$ in $\PG(2,q)$ exist?} and \emph{which nonequivalent KM-arcs of type $t$ in $\PG(2,q)$ exist for given admissable $q$ and $t$?} We give a survey of the known results in Table \ref{overviewKMarcs}.

\begin{table}[ht]
  \begin{tabular}{c|c|c|c|c}
    
     $q$ & $t$ & Condition & Comments & Reference \\\hline\hline
    %$2^{h}$ & $2^{i}$ & $h-i\mid h$ & construction based on $\T:\F_{2^{h}}\to\F_{2^{h-i}}$ & \cite[Sect. 5]{km}\\
    %& & & and an o-polynomial $g$ over $\F_{2^{h-i}}$, & \\
    %& & & translation iff $g(z)=z^{2^{d}}$ with $(d,h-i)=1$ & \\
      $2^{h}$ & $2^{i}$ & $h-i\mid h$ & see Sections \ref{kmrevisited}, \ref{gwrevisited} & \cite[Constr. 3.4(1))]{gw}\\
    
     & & & generalising \cite[Sect. 5]{km} \\\hline
      $2^{h}$ & $2^{i+1}$ & $h-i\mid h$ & see Section \ref{gwrevisited} & \cite[Constr. 3.4(2))]{gw}\\\hline
      $2^{h}$ & $2^{i+m}$ & $h-i\mid h$, a KM-arc & see Section \ref{gwrevisited} & \cite[Constr. 3.4(3))]{gw}\\
       & & of type $2^{m}$ in & & \\
      & & $\PG(2,2^{h-i})$ exists & & \\\hline
      $2^{h}$ & $2^{h-2}$ & $h\geq3$ &  translation KM-arcs & \cite[Sect. 5]{vdd}\\
     & & & &  Sections \ref{newfamily} and \ref{VDD2}\\\hline
     $2^{h}$ & $2^{h-2}$ & $h\geq3$ &  \textbf{new construction} & Section \ref{genervdd}\\
   \hline
     32 & 4 & & see Section \ref{computer} & \cite{kmm} \\\hline
     32 & 8 & &  see Section \ref{computer} & \cite{lim}\\\hline
  
     64 & 8 & &  see Section \ref{computer} & \cite{lim} \\\hline
  \end{tabular}
  \caption{An overview of the known KM-arcs}
  \label{overviewKMarcs}
\end{table}

In this article we will use linear sets to study these problems. We will describe a new family of KM-arcs of type $q/4$, and look at known KM-arcs from this point of view. It was noted a few years ago that KM-arcs together with their $t$-nucleus determine $\F_{2}$-linear sets on each of their $t$-secants, however they are not $\F_{2}$-linear sets themselves. Vandendriessche conjectured in a lecture \cite{petervicenza} that this is always the case.

In the second half of this section we recall the basic information on field reduction and linear sets. In Section \ref{sec:translationKMarcs} we discuss translation KM-arcs. We introduce $i$-clubs and discuss their relationship with KM-arcs. We will prove a characterisation theorem of translation KM-arcs using these $i$-clubs, discuss KM-arcs of type $q/2$ and translation hyperovals from this perspective and describe a family of type $q/4$ using this setting. In Section \ref{sec:revisiting} we will discuss the known KM-arcs from the point of view of linear sets and show that the family of Vandendriessche is a family of translation KM-arcs which can be constructed as in Section \ref{sec:translationKMarcs}. In Section \ref{genervdd} we present a new family of KM-arcs of type $q/4$, including the family of Vandendriessche as well as many examples of non-translation arcs. We end by showing that every translation KM-arc of type $q/4$ is a member of this new family.

\subsection{Linear sets and field reduction}

A {\em $(t-1)$-spread} $\S$ of $\PG(n-1,q)$ is a partition of the point set of $\PG(n-1,q)$ into subspaces of dimension $(t-1)$. It is a classic result of Segre that a $(t-1)$-spread of $\PG(n-1,q)$ can only exist if $t$ divides $n$. The construction of a {\em Desarguesian spread} that follows shows the well-known fact that this condition is also sufficient.

A Desarguesian $(t-1)$-spread of $\PG(rt-1,q)$ can be obtained by applying {\em field reduction} to the points of $\PG(r-1,q^t)$. The underlying vector space of the projective space $\PG(r-1,q^t)$ is $V(r,q^t)$; if we consider $V(r,q^t)$ as a vector space over $\F_q$, then it has dimension $rt$, so it defines a $\PG(rt-1,q)$. In this way, every point $P$ of $\PG(r-1,q^t)$ corresponds to a subspace of $\PG(rt-1,q)$ of dimension $(t-1)$ and it is not hard to see that this set of $(t-1)$-spaces forms a spread of $\PG(rt-1,q)$, which is called a Desarguesian spread. If $U$ is a subset of $\PG(rt-1,q)$, and $\D$ a Desarguesian $(t-1)$-spread, then we define $\B(U):=\{R \in \D \mid U \cap R \neq \emptyset\}$. In this paper, we consider the Desarguesian spread $\D$ as fixed and we identify the elements of $\B(U)$ with their corresponding points of $\PG(r-1, q^t)$. 

Linear sets can be defined in several equivalent ways, but using the terminology introduced here, an \emph{$\F_q$-linear set $S$ of rank $h$} in $\PG(r-1,q^{t})$ is a set of points such that $S=\B(\mu)$, where $\mu$ is an $(h-1)$-dimensional subspace of $\PG(rt-1,q)$. The {\em weight} of a point $P=\B(p)$ of a linear set $\B(\mu)$ equals $\dim(\mu\cap p)+1$.
Following \cite{michel}, a {\em club of rank $h$} is a linear set $\S$ of rank $h$ such that one point of $\S$ has weight $h-1$ and all others have weight $1$.
We define an {\em $i$-club of rank $h$} as a linear set $\mathcal{C}$ of rank $h$ such that one point, called the {\em head} of $\mathcal{C}$ has weight $i$ and all others have weight $1$. With this definition, a usual club of rank $h$ is an $(h-1)$-club. A $1$-club is a {\em scattered linear set}, which is defined to be a linear set of which all points have weight one. We see that the head of a $1$-club is not defined. Note that an $\F_q$-linear $i$-club of rank $h$ in $\PG(1,q^t)$ has size $q^{h-1}+q^{h-2}+\dots+q^{i}+1$. For more information on field reduction and linear sets, we refer to \cite{FQ11}.

\section{Translation KM-arcs and \texorpdfstring{$i$}{i}-clubs}\label{sec:translationKMarcs}

\subsection{A geometric construction for translation arcs}

Let $\D$ be the Desarguesian $(h-1)$-spread in $\PG(3h-1,2)$ corresponding to $\PG(2,2^h)$, let $\ell_\infty$ be the line at infinity of $\PG(2,2^h)$ and let $H$ be the $(2h-1)$-space such that $\ell_\infty=\B(H)$. The points of $\PG(2,2^h)$ that are not on $\ell_\infty$ are the {\em affine} points. 

\begin{theorem}\label{translation}
  Let $\mu$ be an $(h-1)$-space in $H$ such that $\B(\mu)$ is an $i$-club $\C$ of rank $h$ with head $N$ in $\ell_\infty$, and let $\rho\in\D$ be the spread element such that $\B(\rho)=N$. Let $\pi$ be an $h$-space meeting $H$ exactly in $\mu$. Then the point set $\B(\pi)\setminus\C$ together with the points of $\ell_\infty\setminus\C$ forms a translation KM-arc of type $2^i$ in $\PG(2,2^h)$ with axis $\ell_\infty$ and with $2^{i}$-nucleus $N$.
\end{theorem}
\begin{proof}
  We denote $(\B(\pi)\setminus\C)\cup(\ell_\infty\setminus\C)$ by $\A$. %WAAROM STAAT DIT HIER ZO, ZIE IK IETS OVER HET HOOFD?  First, note that for all points $q$ of $\pi\setminus H$, $\B(q)\cap \pi=\{q\}$ and $|\B(\pi)\setminus\C|=2^{h}$. Also note that $\mu\setminus\rho$ contains $(2^{h}-1)-(2^{i}-1)$ points; by the requirement in the theorem those belong to different spread elements. 
  As $\pi$ is an $h$-space that meets $H$, which is spanned by spread elements, in an $(h-1)$-space, a spread element that meets $\pi\setminus \mu$ non-trivially,  meets it in a point. Consequently, $\A$ has $2^h$ affine points. The size of $\C=\B(\mu)$ is $2^{h-1}+\dots+2^i+1$, which implies that $\A$ contains $(2^h+1)-(2^{h-1}+\dots+2^i+1)=2^i$ points of $\ell_\infty$. So in total, $\A$ has $2^h+2^i$ points.%  Hence, $|\ell_\infty\setminus\C|=2^{i}$ and thus $|\A|=2^{h}+2^{i}$. Moreover, $|\A\cap\ell_{\infty}|=2^{i}$
  \par Let $\ell$ be a line in $\PG(2,2^h)$ different from $\ell_{\infty}$, and let $L$ be the $(2h-1)$-space in $\PG(3h-1,2)$ such that $\ell=\B(L)$. If $L\cap H=\rho$, then $L$ contains the $(i-1)$-space $\mu\cap\rho$. Since $L$ contains no other points of $H$ than the points of $\rho$, either $L\cap\pi$ is an $i$-space, or else $L\cap\pi$ equals the $(i-1)$-space $\mu\cap\rho$. In the former case $|\A\cap\ell|=2^{i}$, in the latter case $\ell$ contains no points of $\mathcal{A}$.
  \par If $L\cap H$ is a spread element different from $\rho$, then $L$ meets $\mu$ in a point or $L\cap\mu=\emptyset$. In the former case $\ell$ has no point in common with $\ell_\infty\setminus\C$, and $L$ meets meet $\pi$ in a line or a point, so $\ell\cap(\B(\pi)\setminus\C)$ equals $0$ or $2$. In the latter case $\ell$ has one point in common with $\ell_\infty\setminus\C$, and $L$ meets meet $\pi$ in a point by the Grassmann identity, so $|\ell\cap(\B(\pi)\setminus\C)|$ equals $1$. Consequently, all lines meet $\A$ in $0$, $2$ or $2^{i}$ points, and all lines that meet it in $2^{i}$ points, pass through $N$; $\A$ is a KM-arc of type $2^i$ with $2^i$-nucleus $N$.
 
 We now prove that $\A$ is a translation KM-arc with axis $\ell_{\infty}$. Let $P_{1}$ and $P_{2}$ be two points of $\mathcal{A}\setminus\ell_{\infty}$, and let $Q_{1},Q_{2}\in(\pi\setminus\mu)$ be the points such that $\B(Q_{1})=P_{1}$ and $\B(Q_{2})=P_{2}$. We consider the elation $\gamma$ in the $(2h)$-space $\left\langle H,\pi\right\rangle$ with axis $H$, that maps $Q_{1}$ on $Q_{2}$. This elation induces an elation $\overline{\gamma}$ of $\PG(2,2^{h})$ with axis $\B(H)=\ell_{\infty}$. Note that $\pi$ is fixed by $\gamma$, and hence $\B(\pi)$ is fixed by $\overline{\gamma}$. So $\mathcal{A}$ is fixed by $\overline{\gamma}$. Since $Q^{\gamma}_{1}=Q_{2}$, $P^{\overline{\gamma}}_{1}=P_{2}$. 
\end{proof}

\begin{theorem}\label{translationinverse}
  Every translation KM-arc of type $2^i$ in $\PG(2,2^h)$ can be constructed as in Theorem \ref{translation}.
\end{theorem}
\begin{proof}
  From \cite[Proposition 6.3]{km}, we know that if $\A$ is a translation KM-arc of type $t$ in $\PG(2,q),q=2^h$ with translation line $Z=0$, and $(0,1,0)$ as $t$-nucleus, then the affine points of $\A$ can be written as $(f(t),t,1)$ where $f(z)=\sum_{i=0}^{h-1}\alpha_i z^{2^i}$ with $\alpha_i\in \F_2$.
  \par Now let $\{\omega, \omega^2,\omega^{2^2},\ldots,\omega^{2^{h-1}}\}$ be a normal basis for $\F_{2^h}$ over $\F_2$ and consider field reduction with respect to this basis, i.e. we let the vector $(u,v,w)$ of $\F_{2^h}^3$ correspond to the vector $(u_0,\ldots,u_{h-1};v_0,\ldots,v_{h-1};w_0,\ldots,w_{h-1})$ of $\F_{2}^{3h}$, where $u=\sum_{i=0}^{h-1}u_i\omega^{2^i}$, $v=\sum_{i=0}^{h-1}v_i\omega^{2^i}$ and $w=\sum_{i=0}^{h-1}w_i\omega^{2^i}$. 
  \par Write $1=\sum_{i=0}^{h-1}a_i \omega^{2^i}$, and let $k\in \{0,\dots,h-1\}$ be an index for which $a_k=1$. Let $t\in \F_{2h}=\sum_{i=0}^{h-1}t_i\omega^{2^i}$, then $t^{2^j}=\sum_{i=0}^{h-1}t_{h-j+i}\omega^{2^i}$, where the indices are taken modulo $h$.% we define $t_k=t_{h+k}$ if $k<0$.
  We see that $f(t)=\sum_{j=0}^{h-1}(\sum_{i=0}^{h-1}\alpha_jt_{h-i+j}\omega^{2^i})$, again with the indices taken modulo $h$.
  \par This implies that every point $(f(t),t,1)$, $t\in \F_{2^h}$ is defined by a vector of $\F_{2^{3h}}$ corresponding to a point of $\PG(3h-1,2)$ that belongs to the $h$-dimensional subspace $\pi$ defined by the $2h-1$ equations:
  \begin{align*}
    X_{i}&=\sum_{j=i}^{h-1}\alpha_{j}X_{h-i+j}+\sum_{j=0}^{i-1}\alpha_{j}X_{2h-i+j}, i\in \{0,\dots,h-1\}\\
    X_{2h+j}&=a_jX_{2h+k}, j\in \{0,\dots,h-1\}, j\neq k.
  \end{align*}
  The intersection of $\pi$ with the $(2h-1)$-space $H$ corresponding to the line $z=0$, defined by $X_{2h}=X_{2h+1}=\ldots=X_{3h-1}=0$ satisfies one extra equation, namely $X_{2h+k}=0$, hence, $\pi$ meets $H$ in an $(h-1)$-dimensional space $\mu$.
  \par Since $f$ is an $\F_2$-linear map, the set of directions determined by the set $\{(f(t),t,1)\mid t\in \F_{2^h}\}$ equals $\{(f(z),z,0)| z\in \F_{2^h}\}$. %We only need to show that if $\A$ is a translation KM-arc, then $\{(f(z),z,0)| z\in \F_{2^h}\}$ defines an $i$-club.
  \par If $\A$ is a KM-arc with affine part $\A'$ then it is clear that the set of points of the KM-arc of type $2^i$ on the line at infinity is exactly the set of non-determined directions by $\A'$. The size of this set is $2^i$, which shows that the set $\A'$ determines $2^h-2^i+1$ directions and that $|\B(\mu)|=2^h-2^i+1$. Since we know that the point $(0,1,0)$ lies on all $2^i$-secants to the affine part of $\A$, determined by $\B(\pi)\setminus \B(\mu)$, we obtain that the spread element corresponding to $(0,1,0)$ meets $\mu$ in an $(i-1)$-space. Consequently, all other spread elements that meet $\mu$, meet it in a point, and $\B(\mu)$ is an $i$-club.
\end{proof}

\subsection{The case \texorpdfstring{$i=h-1$}{i=h-1} and projective triads}
A {\em projective triad $\mathcal{T}$} of side $n$ in $\PG(2,q)$ is a set of $3n-2$ points, $n$ on each of $3$ concurrent lines $\ell_0,\ell_1,\ell_2$, with $\ell_0\cap \ell_1\cap \ell_2=\{P\}$ such that a line through a point $R_0\neq P $ of $\mathcal{T}$ on $\ell_0$ and $R_1\neq P$ of $\mathcal{T}$ on $\ell_1$ meets $\ell_2$ in a point $R_2$ of $\mathcal{T}$.

\begin{theorem} \label{clueq}
  All $\F_q$-linear $(h-1)$-clubs of rank $h$ in $\PG(1,q^h)$ are $\PGL$-equivalent to the set $(x,Tr(x))_{\F_{q^h}}$, $x\in \F_{q^h}$, where $Tr$ denotes the trace function from $\F_{q^h}$ to $\F_q$.
\end{theorem}
\begin{proof}
  A set of points skew from $(0,1)$, defined by $q^h-1$ vectors, can be written as $(x,f(x))$, $x\in \F_{q^h}^\ast$ since all these points have a different first coordinate. This implies that every $\F_q$-linear set of rank $h$ skew from $(0,1)$ can be written as $\{(x,f(x))\mid  x\in \F_{q^h}\}$ where $f$ is an $\F_q$-linear map. Clearly, every $(h-1)$-club is $\PGL$-equivalent to an $(h-1)$-club $\S'$ of rank $h$ that has head $(1,0)$ and is skew from $(0,1)$. The spread element corresponding to $(1,0)$ consists of all projective points corresponding to vectors of the form $(x, 0)$, $x\in \F_{q^h}$ and all hyperplanes of this $h$-dimensional subvector space are given by the points $(x,0)$ where $Tr(\alpha x)=0$ for some $\alpha\in \F_{q^h}^\ast$ (see \cite[2.24]{lidl}). The element of $\PGL$ induced by the matrix $\left[\begin{array}{cc}\alpha&0\\0&1\end{array}\right]$ maps the point set $(x,Tr(\alpha x))_{\F_{q^h}}$, $x\in \F_{q^h}^\ast$ onto $(x,Tr(x))_{\F_{q^h}}$, $x\in \F_{q^h}^\ast$ which proves the statement.
\end{proof}

\begin{lemma}\label{vers}
  If $\mu$ and $\mu'$ are subspaces of dimension $h-1$ in $\PG(2h-1,q)$ such that $\B(\mu)=\B(\mu')$ and such that $\mu \cap \mu'$ is a hyperplane of some spread element, then $\mu=\mu'$.
\end{lemma}
\begin{proof}
  Suppose to the contrary that $\mu\neq \mu'$ and let $\pi$ be the $h$-space $\langle \mu,\mu'\rangle$. Since $\B(\mu)=\B(\mu')$, the $q^{h-1}$ spread elements of $\B(\mu)$, different from $\B(\mu\cap\mu')$ meet $\pi$ in a line. Moreover, by counting, we see that every point of $\pi$, not in $\mu\cap \mu'$ lies on such a spread element. But there are $q^{h}$ spread elements different from $\B(\mu\cap \mu')$, each meeting $\pi$ in a point or in a line, a contradiction.
\end{proof}

\begin{corollary}
  There are $q\frac{q^{h}-1}{q-1}$ different $\F_q$-linear $(h-1)$-clubs of rank $h$ with a fixed head in $\PG(1,q^h)$, $h>2$; there are $q(q^{2h}-1)/(q-1)$ different $\F_q$-linear $(h-1)$-clubs of rank $h$ in $\PG(1,q^h)$, $h>2$.
\end{corollary}
\begin{proof}
  Note that the restriction $h>2$ is necessary, since for $h=2$ an $(h-1)$-clubs of rank $h$ is a scattered set and the head of a scattered set is not defined. 
  \par There are $(q^h+1)$ choices for the head of the $(h-1)$-club. For a fixed head, we may fix a hyperplane $\mu$ of the spread element $\rho$ corresponding to it since the elementwise stabiliser of the Desarguesian spread acts transitively on the hyperplanes of one spread element. By the previous lemma, we have that every $(h-1)$-space through $\mu$ different from $\rho$ gives rise to another $(h-1)$-club with head $\B(\rho)$. We know that there are $\frac{q^{h+1}-1}{q-1}-1=q\frac{q^{h}-1}{q-1}$ such $(h-1)$-spaces. This gives in total $(q^h+1)(\frac{q^{h+1}-1}{q-1}-1)=q\frac{q^{2h}-1}{q-1}$ different $\F_q$-linear $(h-1)$-clubs of rank $h$ in $\PG(1,q^h)$.
\end{proof}

\begin{corollary}
  The stabiliser of an $\F_q$-linear $(h-1)$-club in $\PG(1,q^h)$, $h>2$, in $\PGammaL(2,q^h)$, $q=p^t$, $p$ prime, has size $ht q^{h-1}(q-1)$.
\end{corollary}

%DIT ZOU HET VOLGENDE LEMMA ONS TE VER LEIDEN DENK IK, AANGEZIEN WE DAN HET ANDERE LANGE BEWIJS NODIG HEBBEN MET DE PROJECTIE_EQUIVALENTIE. IK HEB BOVENDIEN SCHRIK DAT MICHEL DAT AL ERGENS GEDAAN HEEFT/AAN HET DOEN IS. DOE MIJ ERAAN DENKEN DAT IN IRSEE TE CHECKEN BIJ HEM.
%\begin{lemma} There exist non-equivalent clubs of rank $t$ in $\PG(1,p^h)$ whenever $t<h$.
%\end{lemma}

We will characterise KM-arcs of type $q/2$ by showing that all projective triads are $\F_2$-linear sets. For this, we need the following lemma by Vandendriessche.
\begin{lemma}[{\cite[Corollary 4.4]{vdd}}]
  On $3$ fixed concurrent lines in $\PG(2,q)$, $q$ even, there are $4q-4$ projective triads.
\end{lemma}

\begin{lemma}
  On $3$ fixed concurrent lines of $\PG(2,q)$, $q=2^h$, there are $4q-4$ $\F_2$-linear triads.
\end{lemma}
\begin{proof}
  Fix $3$ concurrent lines, $\ell_1,\ell_2,\ell_3$, with $\ell_1\cap \ell_2=N$ and consider two points $R_1\neq N\in \ell_1$ and $R_2\neq N\in \ell_2$. The point $\left\langle R_1,R_2\right\rangle\cap \ell_3$ is denoted by  $R_3$. Let $S_i$ be the spread element in $\PG(3h-1,2)$ corresponding to $R_i$, $i=1,2,3$, and let $\rho$ be the spread element corresponding to $N$. If $\mu$ is an $h$-space determining (via the construction given in Theorem \ref{translation}) an $\F_2$-linear triad with nucleus $N$ and containing $R_1$ and $R_2$, then clearly, $\mu$ meets $\rho$ in an $(h-2)$-space and contains a transversal line to the regulus through $S_1,S_2,S_3$. If $\B(\mu)$ is a triad through $R_1$, we may always choose $\mu$ in such a way that it contains a fixed point $Q$ of $S_1$. If $\B(\mu)$ contains $R_2$ as well, we see that if $\mu$ contains $Q$, it contains the unique transversal line $\ell$ through $Q$ to the regulus through $S_1,S_2,S_3$. Now there are $2^h-1$ hyperplanes of $\rho$, and each of these hyperplanes defines together with $\ell$ a different $h$-space. It is clear from Lemma \ref{vers} that the defined triads are also different. If we now count triples $(R_1\neq N \text{ on } \ell_1 ,R_2\neq N \text{ on } \ell_2\text{, triad through } R_1 \text{ and }R_2)$, we find that $2^h.2^h.(2^h-1)=X.2^{h-1}.2^{h-1}$, where $X$ is the number of $\F_2$-linear triads. 
\end{proof}

From the previous $2$ lemmas, we immediately get:

\begin{corollary}
  All projective triads are $\F_2$-linear sets.
\end{corollary}

\begin{theorem}
  All projective triads of side $q/2$ in, $q$, $q$ even, are $\PGL$-equivalent.
\end{theorem}
\begin{proof} 
  %We have seen in the previous corollary that all projective triads are $\F_2$-linear sets. Moreover, the points of a projective triad on each of the $q/2$-secants, together with the intersection point of the $3$ concurrent lines, form an $\F_2$-linear set which is an $(h-1)$-club of rank $h$ in $\PG(1,q)$, $q=2^h$. By Lemma \ref{clueq}, we know that all $(h-1)$-clubs are projectively equivalent.
  %%%%
  Let $(X,Y,Z)$ be coordinates for $\PG(2,q)$, let $\mathcal{T}_1$ be a projective triad on the concurrent lines $\ell_1,\ell_2,\ell_3$, where $\ell_1:X=0$, $\ell_2:Y=0$, and $\ell_3:X+Y=0$ and let $\mathcal{T}_2$ be a projective triad on the concurrent lines $m_1,m_2,m_3$. It is clear that there is a collineation $\varphi$ mapping the lines $m_1,m_2,m_3$ onto $\ell_1,\ell_2,\ell_3$ respectively. Let $\mathcal{T}_2'=\varphi(\mathcal{T}_2)$. From Lemma \ref{clueq}, we get that there is an element $\psi\in \PGL(2,q)$ mapping $L_1\cap \mathcal{T}_2'$ onto $L_1\cap \mathcal{T}_1$, let $A=\left(\begin{smallmatrix}a&0\\c&d\end{smallmatrix}\right)$ be a $2\times 2$-matrix corresponding to $\psi$. The $3\times 3$-matrix $\left(\begin{smallmatrix}a&0&0\\0&a&0\\c&c&d\end{smallmatrix}\right)$ defines an element $\bar{\psi}$ of $\PGL(3,q)$ that fixes the points on $\ell_2$ and $\ell_3$, let $\mathcal{T}_2''=\bar{\psi}(\mathcal{T}_2')$. Similarly, we can define an element $\xi$ mapping $\mathcal{T}_2''\cap \ell_2$ onto $\mathcal{T}_1\cap \ell_2$ and fixing the points of $\ell_1$ and $\ell_3$. Since a projective triad is uniquely defined by $\ell_3$ and two sets of $\frac{q}{2}$ points on $\ell_1,\ell_2$ respectively, this implies that $\xi\psi\varphi(\mathcal{T}_2)=\mathcal{T}_1$, and hence $\mathcal{T}_2$ and $\mathcal{T}_1$ are $\PGL$-equivalent.
\end{proof}

\begin{corollary}
  All KM-arcs of type $q/2$ in $\PG(2,q)$, $q$ even, are $\PGL$-equivalent to the example with affine points $\{(x,Tr(x),1)\mid x\in \F_{q}\}$, where $Tr$ denotes the absolute trace function $\F_{q}\to \F_2$.
\end{corollary}

\subsection{Translation hyperovals: \texorpdfstring{$i=1$}{i=1}}
The construction of Theorem \ref{translation} also works for $i=1$. In this case, we  start with a $1$-club in $\PG(1,2^h)$, i.e. a {\em scattered linear set}. The obtained $KM$-arc is an arc of type $2$, which means that it is simply a {\em hyperoval}, and since it is a translation set, we obtain a {\em translation hyperoval}. The correspondence between translation hyperovals and scattered linear sets was already established in \cite[Theorem 2]{glynn}.

%i=0 iets zeggen?

\subsection{A family of translation KM-arcs of type \texorpdfstring{$q/4$}{q/4} in \texorpdfstring{$\PG(2,q)$}{PG(2,q)}: \texorpdfstring{$i=h-2$}{i=h-2}}\label{newfamily}

We construct an $\F_q$-linear $(h-2)$-club of rank $h$ in $\PG(1,q^{h})$, and use Theorem \ref{translation} and this $(h-2)$-club to construct a family of KM-arcs of type $q/4$ in $\PG(2,q)$, $q$ even. Throughout this section we assume $h\geq3$.

\begin{lemma}\label{lshmin3ruimte}
%The set $(t_1\lambda+t_2\lambda^2+\ldots+t_{h-1}\lambda^{h-1}, t_{h-1}+t_{h+1}\lambda)_{\F_{p^h}}$, where $t_i$ are elements in $\F_p$, not all zero and $\lambda$ is a primitive element of $\F_{p^h}$, is an $\F_p$-linear set $\B$ of rank $h$ of size $p^{h-1}+p^{h-2}+1$.
  %INTRODUCTIE LAMBDA
  The set $\C=$
  \[
    \{(t_1\lambda+t_2\lambda^2+\dots+t_{h-1}\lambda^{h-1}, t_{h-1}+t_{h}\lambda)\mid t_{i}\in\F_{q},(t_{1},\dots,t_{h})\neq(0,\dots,0)\}\subseteq\PG(1,q^{h})
  \]
  is an $\F_q$-linear $(h-2)$-club of rank $h$ with head $(1,0)$. Consequently, $|\C|=q^{h-1}+q^{h-2}+1$.
\end{lemma}
\begin{proof}
  The point set $\C$ is determined by the $\F_{q}$-vector set
  \[
    W=\{(0,t_1,t_2,\dots,t_{h-1};t_{h-1},t_{h},0,\dots,0)\mid t_{i}\in\F_{q},(t_{1},\dots,t_{h})\neq(0,\dots,0)\}\subseteq\F^{2h}_{q}
  \]
  It is immediate that $\C$ is an $\F_q$-linear set of rank $h$, as $W\cup\{(0,\dots,0)\}$ is an $h$-dimensional subspace of $\F^{2h}_{q}$.% the $t_i$ are in $\F_p$. By chosing $t_{j}=1$ and $t_{i}=0$ for $i\neq j$, with $j=1,\dots,h-1,h+1$, we find $h$ linearly independent $\F_{p}$-vectors spanning this subspace, so $\B$ has rank $h$.
  \par The projective point $(1,0)$ in $\C$ arises from every vector in $W$ with $t_{h-1}=t_{h}=0$. Together with the zero vector, these form an $(h-2)$-dimensional subspace in $\F^{2h}_{q}$. All other projective points in $\C$ are represented by $q-1$ or by $q^{2}-1$ vectors in $W$. Assume that a point is represented $q^{2}-1$ times, then we can find $t_{1},\dots,t_{h-2},t'_{1},\dots,t'_{h-2}\in\F_{q}$ such that
  \[
    (t_1\lambda+t_2\lambda^2+\ldots+t_{h-2}\lambda^{h-2}+\lambda^{h-1},1)=k(t'_1\lambda+t'_2\lambda^2+\ldots+t'_{h-2}\lambda^{h-2},\lambda)\;,
  \]
  with $k\in\F_{q^{h}}$. Looking at the second coordinate we can see that $k=1/\lambda$. However, $t_1\lambda+t_2\lambda^2+\ldots+t_{h-2}\lambda^{h-2}+\lambda^{h-1}$ and $t'_1+t'_2\lambda+\ldots+t'_{h-2}\lambda^{h-2}$ cannot be the same element of $\F_{q^{h}}$. Hence, all points in $\C$ but $(1,0)$ are represented by precisely $q-1$ vectors in $W$, hence by precisely $1$ point in the projective space $\PG(2h-1,q)$ arising from $\F^{2h}_{q}$. It follows that $\C$ is an $(h-2)$-club.% So, there are $\frac{p^{h}-p^{h-2}}{p-1}=p^{h-1}+p^{h-2}$ points in the linear set $\B$ with $(t_{h-1},t_{h+1})\neq(0,0)$. The lemma follows.
\end{proof}

%DE VOLGENDE REMARK MAG VAN MIJ HELEMAAL WEG (?)
%\begin{remark}   %By applying field reduction to the projective line $\ell=\PG(1,q)$, $q=2^{h}$ and $h\geq3$, we find an $(h-1)$-spread in $\PG(2h-1,2)$. The $(h-2)$-club $\C\subset\ell$ described in Lemma \ref{lshmin3ruimte}, arises from a projective $(h-1)$-space $\mu$, given by the vector set in the statement of the lemma.
%   Using the coordinate functions $X_{0},\dots,X_{2h-1}$ in $\PG(2h-1,p)$, the $h$-space $\mu$ in $\PG(2h-1,p)$ determining the $(h-2)$-club $\C=\B(\mu)$ of the previous lemma has the following equations,: $X_{0}=X_{h-1}+X_{h}=X_{h+2}=X_{h+3}=\dots=X_{2h-1}=0$. %This subspace meets the spread element $\rho$ in an $(h-3)$-space, with $\rho$ the spread element corresponding to the point $N(1,0)$, and meets $2^{h-1}+2^{h-2}$ other spread elements in a point.
% %We can embed the line $\ell$ in the projective plane $\PG(2,q)$ and accordingly the $(h-1)$-spread of $H=\PG(2h-1,2)$ in an $(h-1)$-spread of $\PG(3h-1,2)$, which arises by applying field reduction to $\PG(2,q)$. We choose an $h$-space $\pi$ containing $\mu$, not contained in $H$.
%  %The remaining $2^{h-2}$ spread elements are met by the $(h-2)$-space $\mu'$ that gives rise to $\C'$ and that meets $\rho$ in an $(h-3)$-space. The subspace $\mu'$ is described by $X_{0}+X_{h}=X_{h-2}+c_{h-1}X_{h-1}+X_{h}=X_{h+1}=X_{h+2}=\dots=X_{2h-1}=0$.
%\end{remark}

The existence of an $(h-2)$-club of rank $h$ in $\PG(1,2^h)$ estabilished in Lemma \ref{lshmin3ruimte}, together with Theorem \ref{translation}, yields the following result. 

\begin{theorem}\label{qover4}
  For every $h\geq 3$, there exists a translation $KM$-arc of type $2^{h-2}$ in $\PG(2,2^h)$. In other words, for even $q\geq 8$, there exists a translation set $\A$ of size $q+q/4$ such that every line meets $\A$ in $0$, $2$  or $q/4$ points.
\end{theorem}

As mentioned before, it has been conjectured by Vandendriessche that the points of a KM-arc on a $t$-secant, together with the nucleus, form a linear set.  Let $\A$ be the KM-arc constructed above, with $q/4$-nucleus $N=\B(\rho)$. By the construction described in Theorem \ref{translation} we know that $\{N\}\cup(\A\cap m)$ is a linear set for all four $q/4$-secants $m$ different from $\ell_\infty$. In the following lemma, we check that the conjecture also holds for the line at infinity, which is the fifth $q/4$-secant to $\A$. The lemma will be of use later in Section \ref{VDD2}, when we will show that this family is equivalent to the family given by Vandendriessche.
\begin{lemma}\label{newfamilyremaining}
  Let $\C$ be the $(h-2)$-club in $\ell=\PG(1,q)$, $q=2^{h}$, given in Lemma \ref{lshmin3ruimte}. Then the set $(\ell\setminus\C)\cup\{(1,0)\}$ is an $\F_2$-linear set.
\end{lemma}
\begin{proof}
  Let $f(x)=x^{h}+c_{h-1}x^{h-1}+\ldots+c_{1}x+c_{0}$ be the primitive polynomial of the primitive element $\lambda\in\F_{q}$ that we used in the description of $\C$, so $f(\lambda)=0$. We look at the set
  \begin{multline*}
    \C'=\{(v_{0}+v_1\lambda+v_2\lambda^2+\ldots+v_{h-3}\lambda^{h-3}+(v_{0}+c_{h-1}v_{h-1})\lambda^{h-2}+v_{h-1}\lambda^{h-1}, v_{0})\mid\\ v_i\in\F_2,(v_{0},\dots,v_{h-1})\neq(0,\dots,0)\}\subseteq\ell\;.
  \end{multline*}
  This is clearly an $\F_2$-linear set of rank $h-1$ which contains the projective point $(1,0)$. All other projective points in $\C'$ are the points in the above notation with $v_{0}=1$. Now we prove that $\C\cap\C'=\{(1,0)\}$. Consider a point in the above set with $v_{0}=1$:
  \[
    (1+v_1\lambda+v_2\lambda^2+\ldots+v_{h-3}\lambda^{h-3}+(1+c_{h-1}v_{h-1})\lambda^{h-2}+v_{h-1}\lambda^{h-1},1)\;.
  \]
  The $\F_{q}$-scalar multiples of this vector give rise to the same projective point. The only scalar multiples that we need to look at are
  \begin{align*}
    &\left(c_{0}v_{h-1}+(1+c_{1}v_{h-1})\lambda+(v_{1}+c_{2}v_{h-1})\lambda^{2}+\ldots+(v_{h-3}+c_{h-2}v_{h-1})\lambda^{h-2}+\lambda^{h-1},\lambda\right)\ \text{and}\\
    &\left(1+c_{0}v_{h-1}+(v_{1}+1+c_{1}v_{h-1})\lambda+(v_{2}+v_{1}+c_{2}v_{h-1})\lambda^{2}+\ldots\right.\\ &\qquad\qquad+(v_{h-3}+v_{h-4}+c_{h-3}v_{h-1})\lambda^{h-3}+(1+v_{h-3}+(c_{h-2}+c_{h-1})v_{h-1})\lambda^{h-2}\\
    &\qquad\qquad\left.+(v_{h-1}+1)\lambda^{h-1},\lambda+1\right)\;.
  \end{align*}
  Using field reduction with respect to the basis $\{1,\lambda,\ldots,\lambda^{h-1}\}$ of $\F_{2^h}$ over $\F_2$, these correspond to the vectors
  \begin{align*}
    &(1,v_1,v_2,\ldots,v_{h-3},1+c_{h-1}v_{h-1},v_{h-1};1,0,\dots,0)\;,\\
    &(c_{0}v_{h-1},1+c_{1}v_{h-1},v_{1}+c_{2}v_{h-1},\ldots,v_{h-3}+c_{h-2}v_{h-1},1;0,1,0,\dots,0)\;\text{and}\\
    &(1+c_{0}v_{h-1},v_{1}+1+c_{1}v_{h-1},v_{2}+v_{1}+c_{2}v_{h-1},\ldots,\\&\qquad\qquad v_{h-3}+v_{h-4}+c_{h-3}v_{h-1},1+v_{h-3}+(c_{h-2}+c_{h-1})v_{h-1},v_{h-1}+1;1,1,0,\dots,0)\;.
  \end{align*}
  We can observe that none of these belong to the vector set $W$ from Lemma \ref{lshmin3ruimte}. The lemma follows.
\end{proof}

%\begin{remark}
  %the KM-arcs of type $8$ in $\PG(2,32)$, given in Table \ref{overviewKMarcs}(7), which were described in \cite{lim}, are translation KM-arcs of type $q/4$, with $q=32$. By Theorem \ref{translationinverse} we know they must arise from a $5$-space in $\PG(14,2)$ meeting a $9$-space corresponding to a line in a $4$-space that meets the spread element corresponding to the nucleus in a plane and all other spread elements in a point. The subspace described in Lemma \ref{lshmin3ruimte} and the previous remark (for $h=5$) describes such a $4$-space.
%\end{remark}

\section{The known KM-arcs revisited}\label{sec:revisiting}
\subsection{The arcs of Korchmáros and Mazzocca}\label{kmrevisited}
Korchmáros and Mazzocca constructed a family of KM-arcs of type $2^i$ in $\PG(2,2^h)$, where $h-i | h$. Let $L$ denote the relative trace function from $\F_{2^h}$ to $\F_{2^{h-i}}$ and let $g$ be an $o$-polynomial on $\F_{2^{h-i}}$. Then the authors show that the point set $\{(g(L(x)),x,1)\mid x\in \F_{2^h}\}$ is the affine part of a KM-arc of type $2^i$ \cite{km}. This family contains a subfamily of translation arcs; more precisely, the authors show the following.

\begin{lemma}[{\cite[Proposition 6.4]{km}}]
  The KM-arc with affine points $\{(g(L(x),x,1)\mid x\in \F_{2^h}\}$, where $L$ denotes the relative trace function from $\F_{2^h}$ to $\F_{2^{h-i}}$ is a translation arc of type $2^i$ in $\PG(2,2^h)$ if and only if $g(x)=x^{2^n}$ with $gcd(n,h-i)=1$.
\end{lemma}
Let $\A$ be a translation KM-arc of type $2^i$ in $\PG(2,2^h)$, obtained from the set
\[
  \{(L(x)^{2^n},x,1)\mid x\in \F_{2^h}\},
\]
where $gcd(h-i,n)=1$ and $L$ is the relative trace from $\F_{2^h}$ to $\F_{2^{h-i}}$. By Theorem \ref{translationinverse}, we know that the affine points of $\A$ determine an $i$-club $\C$ on the line at infinity. This yields the following.

\begin{theorem}\label{kmthm}
  The translation KM-arcs of type $2^i$ in $\PG(2,2^h)$ constructed by Korchmáros and Mazzocca can be obtained from the construction of Theorem \ref{translation} by using the $i$-club $\{((L(x)^{2^n},x)\mid x\in \F_{2^h}\}$ on the line at infinity, where $L$ denotes the relative trace from $\F_{2^h}$ to $\F_{2^{h-i}}$ and $gcd(n,h-i)=1$.
\end{theorem}

As the existence and construction of $\F_q$-linear $i$-clubs in $\PG(1,q^h)$ seems a non-trivial problem, we include a construction extending the example of an $i$-club obtained via the construction of Korchmáros and Mazzocca to an $\F_q$-linear $i$-club in $\PG(1,q^h)$.

%of tot de p^n?
\begin{theorem} \label{iclubkm}
  Let $h-i\ |\ h$, $i\geq 1$. The point set $\C=\{(L(x)^{q^n},x)\mid x \in \F_{q^h}\}$, where $L$ denotes the relative trace from $\F_{q^h}$ to $\F_{q^{h-i}}$ and $\gcd(h-i,n)=1$ defines an $\F_q$-linear $i$-club of rank $h$ in $\PG(1,q^h)$.
\end{theorem}
\begin{proof}
  Since $L(ax+y)^{q^n}=aL(x)^{q^n}+L(y)^{q^n}$ for $a\in \F_q$ and $x,y\in \F_{q^h}$, $\C$ defines an $\F_q$-linear set which is clearly of rank $h$. The point $(0,1)$ is defined by all non-zero vectors $(0,t)$ with $L(t)^{q^n}=0$, and there are $(q^i-1)$ such vectors. Moreover, if for some $x,y\in \F_{q^h}$ with $L(y)\neq0$ the vectors $(L(x)^{q^n},x)$ and $(L(y)^{q^n},y)$ define the same projective point, then, by using that $L(x)^{q^n}/L(y)^{q^n}\in \F_{q^{h-i}}$, we obtain that $y=\lambda x$ for some $\lambda\in \F_{q^{h-i}}$. But this in turn yields that $\lambda^{q^n}=\lambda$. Since $\gcd(h-i,n)=1$, this implies that $\lambda\in \F_q$, and hence, that $\C$ is an $\F_q$-linear $i$-club in $\PG(1,q^h)$.
\end{proof}

%\begin{remark} Direct construction with coordinates of the $(h-1)$-space defining $(L(x)^2,x,0)$.
%\end{remark}
\subsection{The arcs of Gács and Weiner}\label{gwrevisited}
In the paper \cite{gw}, Gács and Weiner give a geometric construction of the KM-arcs of type $2^i$ in $\PG(2,2^h)$, $h-i\ |\ h$ described by Korchmáros and Mazzocca (see Section \ref{kmrevisited}). Furthermore, they were able to extend this idea to a construction of KM-arc of type $2^{i+1}$ in $\PG(2,2^h)$, $h-i\ |\ h$. Finally, they gave a recursive construction, starting from a KM-arc of type $2^j$ in $\PG(2,2^{h-i})$, for a KM-arc of type $2^{i+j}$ in $\PG(2,2^{h})$, $h-i\ |\ h$. It is worth mentioning that the authors end their paper with an algebraic description of the three families.

We now state their geometric construction in a different setting using field reduction.

\begin{lemma} \label{lemGW}
  Let $\D'$ be the Desarguesian $(s-1)$-spread in $\PG(3s-1,2^r)$ obtained by applying field reduction to the points of $\PG(2,(2^r)^s)$. Let $\pi$ be a plane of $\PG(3s-1,2^r)$ such that $\B(\pi)$ spans the plane $\PG(2,2^{rs})$. Let $P$ be a point of $\pi$, and let $\rho\in\D'$ be the spread element containing $P$. Let $\mu$ be a hyperplane of $\rho$, not through $P$. Let $H$ be a hyperoval in $\pi$ and let $K$ be the cone with vertex $\mu$ and base $H$, i.e. $K=\bigcup_{Q\in H}\langle Q,\mu\rangle$. Then $\B(K)\setminus \B(\rho)$ is a KM-arc of type $2^{rs-r}$ in $\PG(2,2^{rs})$ if $P\in H$ and of type $2^{rs-r+1}$ if $P\notin H$.
  \par Moreover, if $H'$ is a KM-arc of type $2^j$ in $\pi$ with $2^j$-nucleus $P$ , then $\B(K')=\bigcup_{Q\in H'}\langle Q,\mu\rangle\setminus\B(\rho)$ is a KM-arc of type $2^{rs-r+j}$ in $\PG(2,2^{rs})$.
\end{lemma}
\begin{proof}
  Let $\A$ be the set  $\B(K)\setminus \B(\rho)$ and $\A'=\B(\K')\setminus \B(\rho)$. A line $L$ of $\PG(2,2^{rs})$ corresponds to an $(2s-1)$-dimensional space $\ell$ in $\PG(3s-1,2^r)$, spanned by spread elements.  First suppose that $\rho\subset \ell$, then either $\ell$ meets $\pi$ in the point $P$ or in a line through the point $P$. In the first case, $L$ is skew from $\A$ and $\A'$, in the second, $\ell$ meets $K$ outside $\rho$ either in one or two $(s-1)$-spaces through $\mu$, depending on whether or not $P$ is a point of the hyperoval. This implies that in this case, $L$ meets $\A$ in $2^i$ or $2^{i+1}$ points, with $i=r(s-1)$. As $\L$ meets $K'$ outside $\rho$ in $2^j$ points, we see that $L$ meets $\A'$ in $2^{rs-r+j}$ points. 
  \par Now suppose that $\ell$ does not contain $\rho$. Let $\xi$ be the $(s+1)$-space $\langle \pi,\rho\rangle$, then $\ell$ meets $\xi$ in a line.  Since $K$ is a cone with base a hyperoval, a line skew from the vertex meets this cone in $0$ or $2$ points. Consequently, $L$ contains $0$ or $2$ points of $\A$. Similarly, $K'$ is a cone with base a $KM$-arc of type $2^j$, hence, a line not meeting the subspace spanned by the vertex of the cone and the nucleus of the $KM$-arc meets $K'$ in $0$ or $2$ points.
\end{proof}

Using the algebraic description provided by the authors of \cite{gw}, we see that if we use a translation hyperoval $H$ in the construction of Lemma \ref{lemGW}, we end up with a translation KM-arc, hence, we know that it corresponds to an $i$-club. When the point $P$ is contained in $H$, we find the $i$-club of Theorem \ref{kmthm} but when the point $P$ is not contained in $H$, we can exploit the connection between $i$-clubs and translation KM-arcs to find an $\F_2$-linear $i$-club in $\PG(1,2^h)$ where $h-i+1\ | \ h$.

\begin{theorem}\label{gwthm}
  The translation KM-arcs of type $2^i$, $i=rt-r+1$ in $\PG(2,2^h)$, $h=rt$ constructed by G\'acs and Weiner are $\PGL$-equivalent to the ones obtained from Theorem \ref{translation} by using the $i$-club $\{x_0^{2^n}-ax_0, x_0+\sum_{i=1}^{t-1}x_i\omega^i \mid x\in \F_{2^h}\}$ in $\PG(1,2^h)$, for some $a\in \F_{2^r}^\ast$ and $\gcd(r,n)=1$ where $\{1,\omega,\ldots,\omega^{t-1}\}$, $t>1$ a basis for $\F_{2^{h}}$ over $\F_{2^{r}}$.
\end{theorem}

Also in this case, we can extend the construction of an $i$-club in $\PG(1,2^h)$ where $h-i+1 \ | \ h$ to a construction of an $i$-club in $\PG(1,q^h)$,  $h-i+1 \ | \ h$.

\begin{theorem}\label{iclubgw}
  Let $f: \F_{q^r}\mapsto \F_{q^r}$, $r>1$ be an $\F_q$-linear map such that $\{(x,f(x))\mid x\in \F_{q^r}^\ast\}$ defines a scattered linear set in $\PG(1,q^r)$. Consider the point set $\C=\{(f(x_0)-ax_0,bx_0+\sum_{i=1}^{t-1}x_i\omega^i)\mid x_i \in \F_{q^r}, (x_0\ldots,x_{t-1})\neq (0,\ldots,0)\}$ for some fixed $a,b\in \F_{q^r}$ and $\{1,\omega,\ldots,\omega^{t-1}\}$, $t>1$ a basis for $\F_{q^{rt}}$ over $\F_{q^{r}}$.
  %\par OPMERKING: ik denk dat $x\neq0$ op de eerste lijn, en $(x_{0},\dots,x_{t-1})\neq(0,\dots,0)$ in definitie $\C$. Wat met $r=1$ en $t=1$? AKKOORD
  \begin{itemize}
    \item if $f(x)-ax=0$ does not have non-zero sulutions then $\C$ defines an $\F_q$-linear $r(t-1)$-club of rank $rt$ in $\PG(1,q^{rt})$.
    \item if $f(x)-ax=0$ does have a non-zero solution, then $\C$ defines an $\F_q$-linear $r(t-1)+1$-club of rank $rt$ in $\PG(1,q^{rt})$.
  \end{itemize}
  Hence, there always exist $\F_q$-linear $i$-clubs of rank $h$ in $\PG(1,q^h)$ if $h-i \ | \ h$ and if $h-i+1 \ | \ h$.
\end{theorem}
\begin{proof}
  It is clear from the definition that $\C$ is an $\F_q$-linear set of rank $rt$. If $f(x)-ax=0$ does not have non-zero solutions, then there are $q^{r(t-1)}-1$ vectors that determine the point $(0,1)$.
  \par If $\bar{x}\in\F_{q^{r}}$ is a non-zero solution of $f(x)-ax=0$, then clearly $\lambda\bar{x}$ with $\lambda\in \F_{q}^\ast$ is another non-zero solution. If there would be a non-zero solution $\bar{y}$ to $f(x)-ax=0$, with $\bar{y}$ not an $\F_q$-multiple of $\bar{x}$, then this would imply that the point $(1,a)$ contained in the scattered linear set $\{(x,f(x))\mid x\in \F_{q^r}\}$ is defined by $\bar{x}$ and $\bar{y}$ which are not $\F_q$-multiples, a contradiction. The $q^{r(t-1)}$ values $b\bar{x}+\sum x_i\omega^i$ are different for every choice of $\bar{x}$, so we find in total $q^{r(t-1)+1}-1$ different vectors determining the point $(0,1)$.
  %\par VRAAG. In volgende paragraaf: behoort $\mu$ tot $\F_{q^{r}}$ of tot $\F_{q^{rt}}$?
  \par Now we prove that all other points of $\C$ are determined by a unique vector, up to $\F_q$-multiples. Suppose that $(\mu(f(x_0)-ax_0),\mu(bx_0+\sum x_i\omega^i))=(f(x_0')-ax_0',bx_0'+\sum x_i'\omega^i)$, $\mu \in \F_{q^{rt}}$. We find that $\mu\sum x_i \omega^i=\frac{b}{a}f(x_0')-\mu \frac{b}{a}f(x_0)+\sum x_i' \omega^i$. This implies that $x_i '=\mu x_i$ for $i\in \{1,\ldots,t-1\}$ and that $f(x_0')=\mu f(x_0)$. This in turn implies that $x_0'=\mu x_0$. Finally, since $f$ defines a scattered linear set, we see that $f(\mu x_0')=\mu f(x_0)$ if and only if $\mu \in \F_q$.
  \par We still need to prove the final statement. It is clear that we can always choose $a\in \F_{q^r}$ in such a way that $f(x)-ax=0$ has a non-zero solution and since $|\{f(x)/x\mid x\in \F^{*}_{q^r}\}|<|\F_{q^r}|$, we can also choose an $a\in \F_{q^r}$ such that $f(x)-ax$ has no non-zero solution. Putting $h=rt$ and $i=rt-r+1$ finishes the proof.
\end{proof}

\subsection{The arcs in \texorpdfstring{$\PG(2,32)$}{PG(2,32)} of Limbupasiriporn and Key, MacDonough and Mavron}
\label{computer}

In \cite{lim}, Limbupasiriporn constructs KM-arcs of type $8$ in $\PG(2,32)$. As these are all translation arcs, by Theorem \ref{translationinverse} they can be obtained from the construction of Theorem \ref{translation}. This implies that they correspond to a $3$-club of rank $5$ in $\PG(1,32)$. By computer, we checked that these $3$-clubs are all $\PGL$-equivalent, which implies the following.
\begin{theorem}
  All translation KM-arcs of type $8$ in $\PG(2,32)$ are $\PGL$-equivalent to the example from Theorem \ref{qover4}.
\end{theorem}

Apart from the examples of Vandendriessche, which will be studied in Section \ref{VDD2}, there is one example of a KM-arc in the literature that we did not cover yet, namely, the KM-arc of type $4$ in $\PG(2,32)$ of Key, MacDonough and Mavron \cite{kmm}, which we now denote by $\A_{kmm}$. Note that the parameters $i=2$ and $h=5$ of $\A_{kmm}$ do not fit in one of the infinite classes we have seen before. The arc $\A_{kmm}$ is not a translation arc. We checked by computer that the $4$ points on a $4$-secant to $\A_{kmm}$, together with the $4$-nucleus, do form an $\F_2$-linear set. 

%KLOPT VOLGENDE ZIN WEL?
%But in this example, the size of the set of directions determined by the points of $\A_{kmm}$ on two of its $4$-secants, determines of the choice of the $4$-secants, whereas in the example of Vandendriessche, this is a constant, as we will see in the next section.
%LAAT DIT MAAR WEG

\subsection{The arcs of Vandendriessche}\label{VDD2}

In \cite{vdd} the author described a new family of KM-arcs of type $q/4$ in $\PG(2,q)$. We start by presenting the construction of these examples.

\begin{construction}\label{constvdd}
  Let $\lambda$ be a primitive element of $\F_{q}$, $q=2^{h}$, with minimal polynomial $f$, such that $f(x)=x^{h}+c_{h-3}x^{h-3}+\dots+c_{1}x+c_{0}$, with $c_{i}\in\F_{2}\subset\F_{q}$ for all $i$. Note that $f(x)$ has no terms of degree $h-1$ and $h-2$. Each element $\mu\in\F_{q}$ can uniquely be written as $\mu_{0}+\mu_{1}\lambda+\dots+\mu_{h-1}\lambda^{h-1}$ with $\mu_{i}\in\F_{2}\subset\F_{q}$ for all $i=0,\dots,h-1$.
  \par Let $c\in\F_{2}\subset\F_{q}$ be a parameter. We define the following sets:
  \begin{align*}
    \S_{A}=\{(\mu,1,0)\mid \mu_{h-2}=0,\mu_{h-3}=1\}\;,\\
    \S_{B}=\{(\mu,0,1)\mid\mu_{h-1}=0,\mu_{h-2}=1\}\;,\\
    \S_{C,c}=\{(\mu,1,1)\mid\mu_{h-2}=0,\sum^{h-3}_{i=0}\mu_{i}=c\}\;,\\
    \S_{D,c}=\{(\mu,\lambda,1)\mid\mu_{h-1}+\mu_{h-2}=1,\sum^{h-3}_{i=0}\mu_{i}=c\}\;,\\
    \S_{E,c}=\{(\mu,\lambda^{2},1)\mid\mu_{h-1}=0,\sum^{h-2}_{i=0}\mu_{i}=c\}\;.
  \end{align*}
  Then $\A_{c}=\S_{A}\cup\S_{B}\cup\S_{C,c}\cup\S_{D,c}\cup\S_{E,c}$ is a KM-arc of type $q/4$.
  %Explanation example Peter with coordinates
\end{construction}

\begin{lemma}
  The KM-arcs $\A_{0}$ and $\A_{1}$ from Construction \ref{constvdd} are $\PGL$-equivalent.
\end{lemma}
\begin{proof}
  The collineation induced by the matrix $\left(\begin{smallmatrix}1&0&1\\0&1&0\\0&0&1\end{smallmatrix}\right)$ takes $\A_{0}$ to $\A_{1}$.
\end{proof}

The KM-arcs in $\PG(2,q)$ described in \cite{vdd} thus belong to one orbit of the collineation group $\PGL(3,q)$.

\begin{lemma}\label{vddtranslation}
  The example in Construction \ref{constvdd} is a translation KM-arc.
\end{lemma}
\begin{proof}
  We consider the KM-arc $\A_{0}$ with the points given in Construction \ref{constvdd}. Let $\ell$ be the $\frac{q}{4}$-secant of $\A_{0}$ given by $\lambda Z=Y$. We prove that $\A_{0}$ is translation KM-arc with translation line $\ell$. It is sufficient to chech that for each point $P\in\A_{0}\setminus\ell$ the elation $\varphi_{P}$ mapping $(0,1,1)\in\S_{C,0}$ onto $P$, fixes $\A_{0}$.
  \par Assume $P$ is a point $(\mu,1,0)\in\S_{A}$, then $\varphi_{P}$ is given by
  \[
    \begin{pmatrix}
      \lambda+1&\mu(\lambda+1)&\mu(\lambda+1)\lambda\\
      0&1&\lambda^{2}\\
      0&1&1
    \end{pmatrix}\;.
  \]
  This elation with center $((\lambda+1)\mu,\lambda,1)$ interchanges the lines $Z=0$ and $Y=Z$, and the lines $Y=0$ and $\lambda^{2}Z=Y$. It can be calculated that each point of $\A_{0}\setminus\ell$ is mapped onto another point of $\A_{0}\setminus\ell$. E.g. the point $(x,1,0)$ with $x_{h-2}=0$ and $x_{h-3}=1$ is mapped onto the point $((\lambda+1)(x+\mu),1,1)$, and
  \begin{align*}
    ((\lambda+1)(x+\mu))_{h-2}&=x_{h-3}+x_{h-2}+\mu_{h-3}+\mu_{h-2}=1+0+1+0=0\\
    \sum^{h-3}_{i=0}((\lambda+1)(x+\mu))_{i}&=x_{h-3}+\mu_{h-3}=1+1=0\;,
  \end{align*}
  hence this is a point of $\S_{C,0}$.
  \par If $P$ is a point $(\mu,0,1)\in\S_{B}$, $(\mu,1,1)\in\S_{C,0}$, $(\mu,\lambda^{2},1)\in\S_{E,0}$, then $\varphi_{P}$ is given by
  \begin{align*}
    &\begin{pmatrix}
      \lambda(\lambda+1)&\mu(\lambda+1)&\mu(\lambda+1)\lambda\\
      0&\lambda^{2}&\lambda^{2}\\
      0&1&\lambda^{2}
    \end{pmatrix}\;,
    \quad
    \begin{pmatrix}
      \lambda+1&\mu&\mu\lambda\\
      0&\lambda+1&0\\
      0&0&\lambda+1
    \end{pmatrix}\;\text{and}\\
    &\begin{pmatrix}
      \lambda(\lambda+1)&\mu&\mu\lambda\\
      0&0&\lambda^{2}(\lambda+1)\\
      0&\lambda+1&0
    \end{pmatrix}\;,
  \end{align*}
  respectively. In the calculations to show that each of the points of $\A_{0}\setminus\ell$ is mapped onto a point of $\A_{0}\setminus\ell$, the following equalities are useful:
  \begin{align*}
    \S_{A}&=\{(x,1,0)\mid x_{h-2}=0,x_{h-3}=1\}=\{(y,\lambda+1,0)\mid y_{h-2}=1,\sum^{h-3}_{i=0}x_{i}=1\}\\
    &=\{(z,\lambda(\lambda+1),0)\mid z_{h-1}=1,\sum^{h-2}_{i=0}z_{i}=1\}\;,\\
    \S_{B}&=\{(x,0,1)\mid x_{h-1}=0,x_{h-2}=1\}=\{(y,0,\lambda+1)\mid y_{h-1}=1,\sum^{h-2}_{i=0}y_{i}=1\}\;,\\
    \S_{C,0}&=\{(x,1,1)\mid x_{h-2}=0,\sum^{h-3}_{i=0}x_{i}=0\}=\{(y,\lambda,\lambda)\mid y_{h-1}=0,\sum^{h-2}_{i=0}y_{i}=0\}\;.
  \end{align*}
To obtain these equalities, we used the information on the primitive polynomial of $\lambda$.
\end{proof}

\begin{remark}
  By Theorem \ref{translationinverse} we know that the points on the translation line $\ell$ not belonging to the KM-arc $\A_{0}$, including the nucleus, determine an $(h-2)$-club of rank $h$. So, the point set
  \[
    \ell\setminus\S_{D,0}\cong \{(1,0)\}\cup\{(\mu,1)\mid\mu_{h-1}+\mu_{h-2}=0\vee\sum^{h-3}_{i=0}\mu_{i}=1\}
  \]
  is an $(h-2)$-club of rank $h$.% It is not immediately clear that this is a linear set.
  %\par BETERE BESCRHIJVING VAN DEZE LINEAR SET?
\end{remark}

\begin{theorem}
  The $(h-2)$-club of rank $h$ on the translation line of the KM-arc $\A_{0}$ of type $\frac{q}{4}$ given in Construction \ref{constvdd}, consisting of the nucleus and the points not on the KM-arc, is $\PGL$-equivalent to the $(h-2)$-club $\C$ of rank $h$ given in Lemma \ref{lshmin3ruimte}.
\end{theorem}
\begin{proof}
  Let $\lambda$ be a primitive element of $\F_{q}$, $q=2^{h}$, with minimal polynomial $f$, such that $f(x)=x^{h}+c_{h-3}x^{h-3}+\dots+c_{1}x+c_{0}$, with $c_{i}\in\F_{2}\subset\F_{q}$ for all $i$, as used in Construction \ref{constvdd}.
  \par The set of all points not in the $(h-2)$-club of rank $h$ corresponding to $\A_{0}$ is $\S_{D,0}$, and is equivalent to the set $\{(\mu,1)\mid\mu_{h-1}+\mu_{h-2}=1,\sum^{h-3}_{i=0}\mu_{i}=0\}$. By adding the point $(1,0)$ we find the linear set
  \begin{multline*}
    \left\{\left(\sum^{h-3}_{i=1}t_{i}+t_{1}\lambda+t_{2}\lambda^{2}+\dots+t_{h-3}\lambda^{h-3}+t_{h-2}\lambda^{h-2}+(t_{h-2}+t_{h})\lambda^{h-1},t_{h}\right)\mid\right.\\\left. t_i\in\F_2,(t_{1},t_{2},\dots,t_{h-2},t_{h})\neq(0,\dots,0)\right\}\;,
  \end{multline*}
  which is an $(h-2)$-club of rank $h-1$ with head $(1,0)$. By looking at the collineation induced by $\left(\begin{smallmatrix}1&\lambda^{h-1}\\0&1\end{smallmatrix}\right)$, we can see that this linear set is $\PGL$-equivalent to the linear set
  \begin{multline*}
    \L=\left\{\left(\sum^{h-3}_{i=1}t_{i}+t_{1}\lambda+t_{2}\lambda^{2}+\dots+t_{h-3}\lambda^{h-3}+t_{h-2}\lambda^{h-2}+t_{h-2}\lambda^{h-1},t_{h}\right)\mid\right.\\\left. t_i\in\F_2,(t_{1},t_{2},\dots,t_{h-2},t_{h})\neq(0,\dots,0)\right\}\;.
  \end{multline*}
  By Lemma \ref{newfamilyremaining} and regarding the assumptions on the primive polynomial $f$, the set $\C'=(\ell\setminus\C)\cup\{(1,0\}$ is given by
  \begin{multline*}
    \C'=\{(v_{0}+v_1\lambda+v_2\lambda^2+\ldots+v_{h-3}\lambda^{h-3}+v_{0}\lambda^{h-2}+v_{h-1}\lambda^{h-1}, v_{0})\mid\\ v_i\in\F_2,(v_{0},\dots,v_{h-1})\neq(0,\dots,0)\}\;,
  \end{multline*}
  which is also an $(h-2)$-club of rank $h-1$ with head $(1,0)$. This set is $\PGL$-equivalent to the linear set 
  \begin{multline*}
    \L'=\{(v_1\lambda+v_2\lambda^2+\dots+v_{h-3}\lambda^{h-3}+v_{h-1}\lambda^{h-1}, v_{0})\mid\\ v_i\in\F_2,(v_{0},\dots,v_{h-1})\neq(0,\dots,0)\}\;
  \end{multline*}
  through the collineation $\left(\begin{smallmatrix}1&1+\lambda^{h-2}\\0&1\end{smallmatrix}\right)$. It is sufficient to prove that $\L$ and $\L'$ are $\PGL$-equivalent to conclude that the $(h-2)$-clubs $\ell\setminus\S_{D,0}$ and $\C$ are $\PGL$-equivalent. We note that the collineation induced by $\left(\begin{smallmatrix}1+1/\lambda&0\\0&1\end{smallmatrix}\right)$, maps $\L'$ onto $\L$, which proves the theorem.
\end{proof}

\begin{remark}
  The translation KM-arcs that arise from the $(h-2)$-club of rank $h$ described in Section \ref{newfamily} are thus the KM-arcs that were previously described by Vandendriessche in \cite{vdd}. However, the description presented in Lemma \ref{lshmin3ruimte}  needs no restrictions on the primitive polynomial of the underlying field $\F_{q}$ of the projective plane $\PG(2,q)$.
\end{remark}

\section{A new class of KM-arcs of type \texorpdfstring{$\frac{q}{4}$}{q/4}}\label{genervdd}

We now introduce two properties of KM-arcs of type $\frac{q}{4}$, that are satisfied by the KM-arc of Construction \ref{constvdd}.

\comments{
\begin{definition}\label{starproperty}
  Let $\A$ be a KM-arc of type $\frac{q}{4}$ in $\PG(2,q)$, let $N$ be its nucleus and let $\ell_{0},\dots,\ell_{4}$ be its five $\frac{q}{4}$-secants. Denote $\ell_{i}\cap\A$ by $\S_{i}$, with $i=0,\dots,4$, and denote
  \[
    \D_{ij}=\{\left\langle P,Q\right\rangle\cap\ell_{0}\mid P\in\S_{i},Q\in\S_{j}\}\qquad 1\leq i<j\leq 4\;.
  \]
  The KM-arc $\A$ is said to have property (*) with respect to $\ell_{0}$ if the following requirements are fulfilled.
  \begin{enumerate}
    \item The following equalities are valid:
    \[
      \D_{12}=\D_{34}\qquad\D_{13}=\D_{24}\qquad\D_{14}=\D_{23}\;.
    \]
    %\begin{align*}
      %\{\left\langle P,Q\right\rangle\cap\ell_{0}\mid P\in\S_{1},Q\in\S_{2}\}&=\{\left\langle P,Q\right\rangle\cap\ell_{0}\mid P\in\S_{3},Q\in\S_{4}\}=:\D_{12|34}\;,\\
      %\{\left\langle P,Q\right\rangle\cap\ell_{0}\mid P\in\S_{1},Q\in\S_{3}\}&=\{\left\langle P,Q\right\rangle\cap\ell_{0}\mid P\in\S_{2},Q\in\S_{4}\}=:\D_{13|24}\text{ and }\\
      %\{\left\langle P,Q\right\rangle\cap\ell_{0}\mid P\in\S_{1},Q\in\S_{4}\}&=\{\left\langle P,Q\right\rangle\cap\ell_{0}\mid P\in\S_{2},Q\in\S_{3}\}=:\D_{14|23}\;.
    %\end{align*}
    \item $\D_{12}\cup\{N\}$, $\D_{13}\cup\{N\}$ and $\D_{14}\cup\{N\}$ are $(h-1)$-clubs of rank $h$ with head $N$.
    \item $\left|\D_{12}\cap\D_{13}\right|=\left|\D_{12}\cap\D_{14}\right|=\left|\D_{13}\cap\D_{14}\right|=\frac{q}{4}$.
    \item $\D_{12}\cap\D_{13}\cap\D_{14}=\emptyset$.
  \end{enumerate}
\end{definition}
}

\begin{definition}\label{starproperty}
  Let $\A$ be a KM-arc of type $\frac{q}{4}$ in $\PG(2,q)$, let $N$ be its nucleus and let $\ell_{0},\dots,\ell_{4}$ be its five $\frac{q}{4}$-secants. Denote $\ell_{i}\cap\A$ by $\S_{i}$, with $i=0,\dots,4$, and denote
  \[
    \D_{ij}=\{\left\langle P,Q\right\rangle\cap\ell_{0}\mid P\in\S_{i},Q\in\S_{j}\}\qquad 1\leq i<j\leq 4\;.
  \]
  The KM-arc $\A$ is said to have property (I) with respect to $\ell_{0}$ if the following requirements are fulfilled.
  \begin{enumerate}
    \item The sets $\S_{i}\cup\{N\}$ are $(h-2)$-clubs of rank $h-1$, for all $i=1,\dots,4$.
    \item The following equalities are valid:
    \[
      \D_{12}=\D_{34}\qquad\D_{13}=\D_{24}\qquad\D_{14}=\D_{23}\;.
    \]
    \item $\D_{12}\cup\{N\}$, $\D_{13}\cup\{N\}$ and $\D_{14}\cup\{N\}$ are $(h-2)$-clubs of rank $h-1$ with head $N$.
    %\item $\left|\D_{12}\cap\D_{13}\right|=\left|\D_{12}\cap\D_{14}\right|=\left|\D_{13}\cap\D_{14}\right|=\frac{q}{4}$.
    \item $\D_{12}\cap\D_{13}=\emptyset$, $\D_{12}\cap\D_{14}=\emptyset$ and $\D_{13}\cap\D_{14}=\emptyset$.
  \end{enumerate}
  The KM-arc $\A$ is said to have property (II) with respect to $\ell_{0}$ if the following requirements are fulfilled.
  \begin{enumerate}
    \item The sets $\S_{i}\cup\{N\}$ are $(h-2)$-clubs of rank $h-1$, for all $i=1,\dots,4$.
    \item The following equalities are valid:
    \[
      \D_{12}=\D_{34}\qquad\D_{13}=\D_{24}\qquad\D_{14}=\D_{23}\;.
    \]
    \item $\D_{12}\cup\{N\}$, $\D_{13}\cup\{N\}$ and $\D_{14}\cup\{N\}$ are $(h-1)$-clubs of rank $h$ with head $N$.
    \item $\left|\D_{12}\cap\D_{13}\right|=\left|\D_{12}\cap\D_{14}\right|=\left|\D_{13}\cap\D_{14}\right|=\frac{q}{4}$.
    \item $\D_{12}\cap\D_{13}\cap\D_{14}=\emptyset$.
  \end{enumerate}
\end{definition}

The first requirement in both properties in the previous definition asks that for any partition of the set $\{\S_{1},\S_{2},\S_{3},\S_{4}\}$ in two pairs, these two pairs of point sets determine the same points on $\ell_{0}$. From the other requirements it follows that $\ell_{0}\setminus\{N\}$ is partitioned into four point sets of size $\frac{q}{4}$, namely $\S_{0}$, $\D_{12}\cap\D_{13}$, $\D_{12}\cap\D_{14}$ and $\D_{13}\cap\D_{14}$ in case of property (I), and $\S_{0}$, $\D_{12}$, $\D_{13}$ and $\D_{14}$ in case of property (II). So, the properties (I) and (II) look like two flavours of the same property.

\begin{lemma}\label{translationprop1}
  A translation KM-arc of type $\frac{q}{4}$ in $\PG(2,q)$, has property (I) with respect to its translation line.
\end{lemma}
\begin{proof}
  Let $\A$ be a translation KM-arc of type $\frac{q}{4}$ in $\PG(2,q)$, with $\ell$ its translation line. Denote its nucleus by $N$. Let $\D$ be the Desarguesian $(h-1)$-spread in $\PG(3h-1,2)$ corresponding to $\PG(2,2^h)$, let $H$ be the $(2h-1)$-space such that $\ell=\B(H)$, and let $\rho\in\D$ be the spread element such that $N=\B(\rho)$. By Theorem \ref{translationinverse} we know that there is an $(h-1)$-space $\mu\subset H$ and an $h$-space $\pi$ meeting $H$ exactly in $\mu$ such that $\B(\mu)=\ell\setminus\A$ and $\B(\pi)=\A\setminus\ell$. Moreover $\mu\cap\rho=\sigma$ is an $(h-3)$-space, and $\mu$ meets all other spread elements in either a point or an empty space.
  \par There are four $(h-2)$-spaces through $\sigma$ that are not contained in $H$. We denote them by $\sigma_{i}$, $i=1,\dots,4$. The point set $\B(\sigma_{i})=\A_{i}\cup\{N\}$ is clearly an $(h-2)$-club of rank $h-1$, $i=1,\dots,4$.%contains $\frac{q}{4}$ points of $\A$ and is contained in the $\frac{q}{4}$-secant $\ell_{i}=\B(\left\langle\mu,\sigma_{i}\right\rangle)$, $i=1,\dots,4$.
  \par Using the notation from Definition \ref{starproperty} we see that $\D_{i,j}$ equals $\B(\left\langle\sigma_{i},\sigma_{j}\right\rangle\cap H)$, $1\leq i<j\leq4$.
  \par Consider the quotient geometry of $\sigma$ in $\pi$; this is a Fano plane $\mathcal{F}$ in which $H\cap \pi$ is a line $\ell_{H}$. The $(h-2)$-spaces $\sigma_{1},\dots,\sigma_{4}$ correspond to the points of $\mathcal{F}$ not on $\ell_{H}$. It is clear that $\D_{ij}$ corresponds to the intersection of a line in $\mathcal{F}$ with $\ell_H$. This implies that $\D_{12}=\D_{34}$, $\D_{13}=\D_{24}$ and $\D_{14}=\D_{23}$. Moreover, $\D_{12}$, $\D_{13}$ and $\D_{14}$ correspond to the three $(h-2)$-spaces through $\sigma$ in $\mu$. The two other requirements for property (I) follow immediately.
%
%  
%  
%    We look at the $(h-1)$-spaces $\left\langle\sigma_{i},\sigma_{j}\right\rangle$ through the quotient geometry of $\sigma$ in $\pi$; this is a Fano plane $\mathcal{F}$ in which $H$ is a line $\ell_{H}$. The $(h-2)$-spaces $\sigma_{1},\dots,\sigma_{4}$ are the points of $\mathcal{F}$ not on $\ell_{H}$. Using this quotient geometry we see immediately that $\D_{12}=\D_{34}$, $\D_{13}=\D_{24}$ and $\D_{14}=\D_{23}$. Moreover, $\D_{12}$, $\D_{13}$ and $\D_{14}$ are the three $(h-2)$-spaces through $\sigma$ in $\mu$. The two other requirements for property (I) follow immediately.
\end{proof}

By Lemma \ref{vddtranslation} we know that the KM-arc of type $\frac{q}{4}$ given in Construction \ref{constvdd} is a translation KM-arc, so it has property (I) with respect to its translation line. We show that is has property (II) with respect to another $\frac{q}{4}$-secant.

\begin{lemma}
  Let $\A_{0}\subset\PG(2,q)$ be the KM-arc of type $\frac{q}{4}$ described in Construction \ref{constvdd}, let $N$ be its nucleus, and let $\ell$ be the line $Z=0$, which is a $\frac{q}{4}$-secant of $\A$. Then $\A$ has property (II) with respect to $\ell$. % so $\S_{A}\subset\ell$. Then,
  %\begin{align*}
    %\{\left\langle P,Q\right\rangle\cap\ell\mid P\in\S_{B},Q\in\S_{C,0}\}&=\{\left\langle P,Q\right\rangle\cap\ell\mid P\in\S_{D,0},Q\in\S_{E,0}\}=:\D_{BC|DE}\;,\\
    %\{\left\langle P,Q\right\rangle\cap\ell\mid P\in\S_{B},Q\in\S_{D,0}\}&=\{\left\langle P,Q\right\rangle\cap\ell\mid P\in\S_{C,0},Q\in\S_{E,0}\}=:\D_{BD|CE}\;,\\
    %\{\left\langle P,Q\right\rangle\cap\ell\mid P\in\S_{B},Q\in\S_{E,0}\}&=\{\left\langle P,Q\right\rangle\cap\ell\mid P\in\S_{C,0},Q\in\S_{D,0}\}=:\D_{BE|CD}\;,
  %\end{align*}
  %and each of these three sets contains $\frac{q}{2}$ points. The union of any of these three sets and $N$ is a linear set. Moreover,
  %\[
    %\left|\D_{BC|DE}\cap\D_{BD|CE}\right|=\left|\D_{BC|DE}\cap\D_{BE|CD}\right|=\left|\D_{BD|CE}\cap\D_{BE|CD}\right|=\frac{q}{4}\;,
  %\]
  %and $\D_{BC|DE}\cap\D_{BD|CE}\cap\D_{BE|CD}=\emptyset$
  %Description of the behaviour of the points determined by the affine points on two lines.
\end{lemma}
\begin{proof}
  Note that the line $\ell$ contains the point set $\S_{A}$. First, it is a direct observation that $\S_{B}\cup\{N\}$, $\S_{C,0}\cup\{N\}$, $\S_{D,0}\cup\{N\}$ and $\S_{E,0}\cup\{N\}$ are $(h-2)$-clubs of rank $h-1$, given the form in which they are presented. Now, it is a straighforward calculation to see that
  \begin{align*}
    \{\left\langle P,Q\right\rangle\cap\ell\mid P\in\S_{B},Q\in\S_{C,0}\}&=\{(x,1,0)\mid x_{h-2}=1\}\text{ and }\\
    \{\left\langle P,Q\right\rangle\cap\ell\mid P\in\S_{D,0},Q\in\S_{E,0}\}&=\{(y,\lambda^{2}+\lambda,0)\mid \sum_{i=0}^{h-1} y_{i}=1\}\;.
  \end{align*}
  Another calculation learns that these two point sets are equal. Hereby, we used the information on the primitive polynomial of $\lambda$ given in Construction \ref{constvdd}. We denote this point set by $\D_{BC|DE}$. Analogously we find that
  \begin{align*}
    \D_{BD|CE}&=\{(x,1,0)\mid x_{h-3}+x_{h-2}=0\}=\{(y,\lambda,0)\mid y_{h-2}+y_{h-1}=0\}\\
    &=\{(z,\lambda^{2}+1,0)\mid \sum^{h-2}_{i=0}z_{i}=0\}\;,\\
    \D_{BE|CD}&=\{(x,1,0)\mid x_{h-3}=0\}=\{(y,\lambda^{2},0)\mid y_{h-1}=0\}=\{(z,\lambda+1,0)\mid \sum^{h-3}_{i=0}z_{i}=0\}\;.
  \end{align*}
  It is immediate that $\D_{BC|DE}\cup\{N\}$, $\D_{BD|CE}\cup\{N\}$ and $\D_{BE|CD}\cup\{N\}$ are $(h-1)$-clubs of rank $h$.
  \par Now, we look at the equations $x_{h-2}=1$, $x_{h-3}+x_{h-2}=0$ and $x_{h-3}=0$. Any two of them are independent, but no element of $\F_{q}$ can satisfy all three. So, the intersections $\D_{BC|DE}\cap\D_{BD|CE}$, $\D_{BC|DE}\cap\D_{BE|CD}$ and $\D_{BD|CE}\cap\D_{BE|CD}$ have size $\frac{q}{4}$ and the intersection $\D_{BC|DE}\cap\D_{BD|CE}\cap\D_{BE|CD}$ is empty. It follows that $\A_{0}$ has property (II) with respect to $\ell$.
\end{proof}

\begin{remark}
  The line $\ell$ is not necessarily the unique $\frac{q}{4}$-secant such that $\A_{0}$ has property (II) with respect to it. E.g. for $q=32$, the KM-arc $\A_{0}$ has property (II) with respect to each of the four $\frac{q}{4}$-secants, different from the translation line.
  %\par HOE ZIT HET BIJ KM en GW?
\end{remark}

We now present an easy to prove lemma on systems of equations with trace functions. We will need this lemma throughout the next theorems.
\comments{
\begin{lemma}\label{tracesysteem}
  Let $\T$ be the absolute trace function $\F_{q}\to\F_{2}$, $q$ even. Let $k_{1},k_{2},k_{3}$ be different elements of $\F_{q}$ and let $c_{1},c_{2},c_{3}$ be elements in $\F_{2}$.
  \begin{itemize}
    \item The system
    \[
      \begin{cases}
        \T(k_{1}x)=c_{1}\\
        \T(k_{2}x)=c_{2}
      \end{cases}
    \]
    has $\frac{q}{4}$ solutions.
    \item If $k_{1}+k_{2}+k_{3}\neq0$, then the system
    \[
      \begin{cases}
        \T(k_{1}x)=c_{1}\\
        \T(k_{2}x)=c_{2}\\
        \T(k_{3}x)=c_{3}
      \end{cases}
    \]
    has $\frac{q}{8}$ solutions.
  \end{itemize}
\end{lemma}

OF:
}
\begin{lemma}\label{tracesysteemalg}
  %Let $\T$ be the trace function $\F_{q}\to\F_{2}$, $q$ even. Let $k_{1},\dots,k_{m}$ be $m$ elements of $\F_{q}$ such that $\{k_{1},\dots,k_{m'}\}$ is a basis of the $\F_{2}$-subvector space of $\F_{q}$, generated by $k_{1},\dots,k_{m}$, and let $a_{i,j}\in\F_{2}$ be such that $k_{i}=\sum^{m'}_{j=1}a_{i,j}k_{j}$ for all $i=m'+1,\dots,m$ and $j=1,\dots,m'$. Let $c_{1},\dots,c_{m}$ be elements in $\F_{2}$. We consider the system of equations
  Let $\T$ be the absolute trace function $\F_{q}\to\F_{2}$, $q$ even. Let $k_{1},\dots,k_{m}$ be $m$ elements of $\F_{q}$ and let $c_{1},\dots,c_{m}$ be elements in $\F_{2}$. We consider the system of equations
  \[
    \begin{cases}
      \T(k_{1}x)=c_{1}\\
      \vdots\\
      \T(k_{m}x)=c_{m}
    \end{cases}\;.
  \]
  Up to the ordering of the equations, we can assume that $\{k_{1},\dots,k_{m'}\}$ is a basis of the $\F_{2}$-subvector space of $\F_{q}$, generated by $k_{1},\dots,k_{m}$. Let $a_{i,j}\in\F_{2}$ be such that $k_{i}=\sum^{m'}_{j=1}a_{i,j}k_{j}$ for all $i=m'+1,\dots,m$ and $j=1,\dots,m'$.
  \par If $c_{i}=\sum^{m'}_{j=1}a_{i,j}c_{j}$ for all $i=m'+1,\dots,m$, then this system has $\frac{q}{2^{m'}}$ solutions, otherwise is has no solutions.
  \par In particular, the system has $\frac{q}{2^{m}}$ solutions if $\{k_{1},\dots,k_{m}\}$ is an $\F_{2}$-linear independent set.
\end{lemma}

We now describe a new family of KM-arcs of type $\frac{q}{4}$ that have property (I) or (II) with respect to one of its $\frac{q}{4}$-secants.

\begin{theorem}\label{ourconstruction}
  Let $\T$ be the absolute trace function from $\F_{q}$ to $\F_{2}$, $q=2^{h}$. Choose $\alpha,\beta\in\F_{q}\setminus\{0,1\}$ such that $\alpha\beta\neq1$ and define
  \[
    \gamma=\frac{\beta+1}{\alpha\beta+1}\;,\quad\xi=\alpha\beta\gamma\;.
  \]
  Now choose $a,b\in\F_{2}\subset\F_{q}$, and define the following sets
  \begin{align*}
    \S_{0}&:=\left\{(z,1,0)\mid z\in \F_q,\T(z)=0,\T\left(\frac{z}{\alpha}\right)=a\right\}\;,\\
    \S_{1}&:=\left\{(z,0,1)\mid z\in \F_q,\T(z)=0,\T\left(\frac{z}{\alpha\gamma}\right)=0\right\}\;,\\
    \S_{2}&:=\left\{(z,1,1)\mid z\in \F_q,\T(z)=1,\T\left(\frac{z}{\alpha\beta}\right)=b\right\}\;,\\
    \S_{3}&:=\left\{(z,\gamma,1)\mid z\in \F_q,\T\left(\frac{z}{\alpha\gamma}\right)=a+1,\T\left(\frac{z}{\xi}\right)=b+1\right\}\;,\\
    \S_{4}&:=\left\{(z,\beta+1,1)\mid z\in \F_q,\T\left(\frac{z}{\alpha\beta}\right)=a+b+1,\T\left(\frac{z}{\xi}\right)=b\right\}\;.
  \end{align*}
  Then, $\A=\cup^{4}_{i=0}\S_{i}$ is a KM-arc of type $\frac{q}{4}$ in $\PG(2,q)$, and it has property (I) or (II) with respect to $Z=0$.
\end{theorem}
\begin{proof}
  First note that $\xi+\beta+\gamma=1$. This observation is useful during the proof.
  \par Clearly the lines $\ell_{0}:Z=0$, $\ell_{1}:Y=0$, $\ell_{2}:Y=Z$, $\ell_{3}:Y=\gamma Z$ and $\ell_{4}:Y=(\beta+1)Z$ each contain $\frac{q}{4}$ points of $\A$. For $\A$ to be a KM-arc it is sufficient to check whether all other lines that contain at least two points of $\A$ contain precisely two points of $\A$.
  \par The line $\left\langle(x,1,0),(y,0,1)\right\rangle$ where $x,y\in \F_q$, with $\T(x)=0$, $\T\left(\frac{x}{\alpha}\right)=a$, $\T(y)=0$ and $\T\left(\frac{y}{\alpha\gamma}\right)=0$ contains a point of both $\S_{0}$ and $\S_{1}$, and contains the points $(x+y,1,1)$, $(x\gamma+y,\gamma,1)$ and $(x(\beta+1)+y,\beta+1,1)$ on $\ell_{2},\ell_{3},\ell_{4}$ respectively. Now,
  \begin{align*}
    \T(x+y)&=\T(x)+\T(y)=0+0=0\neq1\;,\\
    \T\left(\frac{x\gamma+y}{\alpha\gamma}\right)&=\T\left(\frac{x}{\alpha}\right)+\T\left(\frac{y}{\alpha\gamma}\right)=a+0=a\neq a+1\;,\\
    \T\left(\frac{x(\beta+1)+y}{\alpha\beta}\right)+\T\left(\frac{x(\beta+1)+y}{\xi}\right)&=\T\left(\frac{x}{\alpha}\right)+\T\left(\frac{x\gamma}{\alpha\beta\gamma}\right)\\&\qquad+\T\left(\frac{(1+\beta)x}{\alpha\beta\gamma}\right)+\T\left(\frac{y(\gamma+1))}{\alpha\beta\gamma}\right)\\
    &=a+\T\left(\frac{(1+\beta+\gamma)x}{\alpha\beta\gamma}\right)+\T\left(\frac{y(\alpha+1))}{\alpha\gamma(\alpha\beta+1)}\right)\\
    &=a+\T(x)+\T(y)+\T\left(\frac{y}{\alpha\gamma}\right)\\
    &=a+0+0+0=a\neq a+1=(a+b+1)+b\;.
  \end{align*}  
  \par So, $(x+y,1,1),(x\gamma+y,\gamma,1),(x(\beta+1)+y,\beta+1,1)\notin\A$. Analogously, on the line through $(x,1,0)\in\S_{0}$ and $(y,1,1)\in\S_{2}$ the points $(x(\gamma+1)+y,\gamma,1)$ and $(x\beta+y,\beta+1,1)$ are not contained in $\S_{3}$ and $\S_{4}$, respectively, since
  \begin{align*}
    \T\left(\frac{x(\gamma+1)+y}{\alpha\gamma}\right)+\T\left(\frac{x(\gamma+1)+y}{\xi}\right)&=\T\left(\frac{(x(\gamma+1)+y)(\beta+1)}{\alpha\beta\gamma}\right)\\
    &=\T\left(\frac{x}{\alpha}\right)+\T\left(\frac{x(1+\beta+\gamma)}{\alpha\beta\gamma}\right)\\
    &\qquad+\T\left(\frac{y(\alpha\beta+1)}{\alpha\beta}\right)\\
    &=a+\T(x)+\T(y)+\T\left(\frac{y}{\alpha\beta}\right)\\
    &=a+b+1\neq a+b=(a+1)+(b+1)\;,\\
    \T\left(\frac{x\beta+y}{\alpha\beta}\right)&=\T\left(\frac{x}{\alpha}\right)+\T\left(\frac{y}{\alpha\beta}\right)=a+b\neq a+b+1\;.
  \end{align*}
  On the line through $(x,1,0)\in\S_{0}$ and $(y,\gamma,1)\in\S_{3}$ the point $(x(\beta+\gamma+1)+y,\beta+1,1)$ is not contained in $\S_{4}$ since
  \[
    \T\left(\frac{x(\beta+\gamma+1)+y}{\xi}\right)=\T\left(\frac{x(\beta+1)\alpha \beta}{(\alpha\beta+1)\alpha\beta\gamma}\right)+\T\left(\frac{y}{\xi}\right)=\T(x)+b+1=b+1\neq b\;.
  \]
  On the line through $(x,0,1)\in\S_{1}$ and $(y,1,1)\in\S_{2}$ the points $(x(\gamma+1)+y\gamma,\gamma,1)$ and $(x\beta+y(\beta+1),\beta+1,1)$ are not contained in $\S_{3}$ and $\S_{4}$, respectively, since
  \begin{align*}
    \T\left(\frac{x(\gamma+1)+y\gamma}{\xi}\right)&=\T\left(\frac{x(\alpha+1)}{\alpha\gamma(\alpha\beta+1)}\right)+\T\left(\frac{y}{\alpha\beta}\right)=\T\left(\frac{x(\alpha\gamma+1)}{\alpha\gamma}\right)+b\\&=\T(x)+\T\left(\frac{x}{\alpha\gamma}\right)+b=b\neq b+1\;,\\
    \T\left(\frac{x\beta+y(\beta+1)}{\xi}\right)&=\T\left(\frac{x}{\alpha\gamma}\right)+\T\left(\frac{y(\alpha\beta+1)}{\alpha\beta}\right)=\T\left(y\right)+\T\left(\frac{y}{\alpha\beta}\right)=b+1\neq b\;.
  \end{align*}
  On the line through $(x,0,1)\in\S_{1}$ and $(y,\gamma,1)\in\S_{3}$ the point $(x\frac{\beta+\gamma+1}{\gamma}+y\frac{\beta+1}{\gamma},\beta+1,1)$ is not contained in $\S_{4}$ since
  \begin{align*}
    \T\left(\frac{\frac{x(\beta+\gamma+1)}{\gamma}+\frac{y(\beta+1)}{\gamma}}{\alpha\beta}\right)&=\T\left(\frac{x\xi}{\xi}\right)+\T\left(\frac{y(\beta+1)}{\alpha\beta\gamma}\right)=\T(x)+\T\left(\frac{y}{\alpha\beta}\right)+\T\left(\frac{y}{\xi}\right)\\
    &=a+b\neq a+b+1\;.
  \end{align*}
  Finally, on the line through $(x,1,1)\in\S_{2}$ and $(y,\gamma,1)\in\S_{3}$ the point $(x(1+\frac{\beta}{\gamma+1})+y\frac{\beta}{\gamma+1},\beta+1,1)$ is not contained in $\S_{4}$ since
  \begin{align*}
    \T\left(\frac{x(1+\frac{\beta}{\gamma+1})+y\frac{\beta}{\gamma+1}}{\alpha\beta}\right)+\T\left(\frac{x(1+\frac{\beta}{\gamma+1})+y\frac{\beta}{\gamma+1}}{\xi}\right)&=\T\left(\frac{x(1+\beta+\gamma)+y\beta}{\xi}\right)\\
    &=\T\left(\frac{x(\beta+1)}{\gamma(\alpha\beta+1)}\right)+\T\left(\frac{y}{\alpha\gamma}\right)\\
    &=\T(x)+a+1=a\neq a+1\;.
  \end{align*}
  This finishes the proof that $\A$ is a KM-arc of type $\frac{q}{4}$. We now show that it has property (I) or (II) with respect to the line $\ell_{0}$. The first requirement, which is the same for both property (I) and (II), is fulfilled: since it follows directly from the definition that $\S_{i}\cap\{N\}$ is an $(h-2)$-club of rank $h-1$. Now, we need to distinguish between two cases.
  \par First we assume that $\beta=\gamma$, which is equivalent to $\alpha\beta^{2}=1$. Using the notation from Definition \ref{starproperty}, we see that
  \begin{align*}
    \D_{12}&=\{(z,1,0)\mid z\in \F_q,\T(z)=1,\T\left(\frac{z}{\alpha\beta}\right)=b\}\text{ and}\\
    \D_{34}&=\{(z,1+\beta+\gamma,0)\mid z\in \F_q,\T\left(\frac{z}{\alpha\beta}\right)=b,\T\left(\frac{z}{\xi}\right)=1\}\\
    &=\{(z,1,0)\mid z\in \F_q,\T\left(\frac{z}{\alpha\beta}\right)=b,\T(z)=1\}\;,\\
    \D_{13}&=\{(z,\gamma,0)\mid z\in \F_q,\T(z)=b+1,\T\left(\frac{z}{\alpha\beta}\right)=a+1\}\text{ and}\\
    \D_{24}&=\{(z,\beta,0)\mid z\in \F_q,\T(z)=b+1,\T\left(\frac{z}{\alpha\beta}\right)=a+1\}\;,\\
    \D_{14}&=\{(z,\beta+1,0)\mid z\in \F_q,\T(z)=b,\T\left(\frac{z}{\alpha\beta}\right)=a+b+1\}\text{ and}\\%=\{(z,1,0)\mid\T\left(\frac{z(a+1)}{\alpha}\right)=a+1\}
    \D_{23}&=\{(z,\gamma+1,0)\mid z\in \F_q,\T(z)=b,\T\left(\frac{z}{\alpha\beta}\right)=a+b+1\}\;.%=\{(z,1,0)\mid\T\left(\frac{z(a+1)}{\alpha}\right)=a+1\}
  \end{align*}
  It follows that $\D_{12}=\D_{34}$, $\D_{13}=\D_{24}$ and $\D_{14}=\D_{23}$ since $\beta=\gamma$ and that $\D_{12}\cup\{N\}$, $\D_{13}\cup\{N\}$ and $\D_{14}\cup\{N\}$ are $(h-2)$-clubs of rank $h-1$ with head $N$. So, the second and the third requirement of property (I) are fulfilled. Using $\a\beta^{2}=1$, we find that
  \begin{align*}
    \D_{12}&=\{(z,1,0)\mid z\in \F_q,\T(z)=1,\T\left(\frac{z}{\alpha\beta}\right)=b\}=\{(z,1,0)\mid\T(z)=1,\T\left(\beta z\right)=b\}\\
    \D_{13}&=\{(z,\beta,0)\mid z\in \F_q,\T(z)=b+1,\T\left(\frac{z}{\alpha\beta}\right)=a+1\}\\
    &=\{(z,1,0)\mid z\in \F_q,\T(\beta z)=b+1,\T\left(\frac{z}{\alpha}\right)=a+1\}\\
    \D_{14}&=\{(z,\beta+1,0)\mid z\in \F_q,\T(z)=b,\T\left(\frac{z}{\alpha\beta}\right)=a+b+1\}\\
    &=\{(z,1,0)\mid z\in \F_q,\T((\beta+1)z)=b,\T\left(\beta(\beta+1)z\right)=a+b+1\}\\
    &=\{(z,1,0)\mid z\in \F_q,\T(\beta z)+\T(z)=b,\T\left(\beta z\right)+\T\left(\frac{z}{\alpha}\right)=a+b+1\}
  \end{align*}
  and from this observation the final requirement of property (I), namely that $\D_{12}\cap\D_{13}=\D_{12}\cap \D_{14}=\D_{13}\cap \D_{14}=\emptyset$ follows. So, in this case $\A$ has property (I) with respect to $\ell_{0}$.
  \par Now, we assume that $\beta\neq\gamma$. We find
  \begin{align*}
    \D_{12}&=\{(z,1,0)\mid z\in \F_q,\T(z)=1\}\text{ and}\\
    \D_{34}&=\{(z,1+\beta+\gamma,0)\mid z\in \F_q,\T\left(\frac{z}{1+\beta+\gamma}\right)=1\}=\{(z,1,0)\mid\T(z)=1\}\;,\\
    \D_{13}&=\{(z,\gamma,0)\mid z\in \F_q,\T\left(\frac{z}{\alpha\gamma}\right)=a+1\}=\{(z,1,0)\mid\T\left(\frac{z}{\alpha}\right)=a+1\}\text{ and}\\
    \D_{24}&=\{(z,\beta,0)\mid z\in \F_q,\T\left(\frac{z}{\alpha\beta}\right)=a+1\}=\{(z,1,0)\mid\T\left(\frac{z}{\alpha}\right)=a+1\}\;,\\
    \D_{14}&=\{(z,\beta+1,0)\mid z\in \F_q,\T\left(\frac{z(\alpha+1)}{\alpha(\beta+1)}\right)=a+1\}\\
    &=\{(z,1,0)\mid\T\left(\frac{z(\alpha+1)}{\alpha}\right)=a+1\}\text{ and}\\
    \D_{23}&=\{(z,\gamma+1,0)\mid z\in \F_q,\T\left(\frac{z(\alpha+1)}{\alpha(\gamma+1)}\right)=a+1\}\\
    &=\{(z,1,0)\mid\T\left(\frac{z(\alpha+1)}{\alpha}\right)=a+1\}\;.
  \end{align*}
  We observe that $\D_{12}=\D_{34}$, $\D_{13}=\D_{24}$ and $\D_{14}=\D_{23}$, so the second requirement of property (II) is fulfilled. It is clear that the sets $\D_{12}\cup\{N\}$, $\D_{13}\cup\{N\}$ and $\D_{14}\cup\{N\}$, with $N(1,0,0)$ the nucleus of $\A$, are $(h-1)$-clubs of rank $h$, so also the third requirement is fulfilled. By Lemma \ref{tracesysteemalg}, the intersections $\D_{12}\cap\D_{13}$, $\D_{12}\cap\D_{14}$ and $\D_{13}\cap\D_{14}$ each contain $\frac{q}{4}$ points. The final requirement, that $\D_{12}\cap\D_{13}\cap\D_{14}=\emptyset$ follows from the observation that 
  \[
    z+\frac{z}{\alpha}=\frac{(\alpha+1)z}{\alpha}\;.\qedhere
  \]
\end{proof}

%\begin{theorem}\label{ourconstruction} $a,b\alpha,\beta,\gamma,\delta$ as follows:
 %$$\alpha+\beta+\gamma=1$$
 %$$\delta=\beta+1$$
 %$$\alpha=\frac{a\beta(\beta+1)}{a\beta+1}$$
 %$$\alpha,\beta,\gamma,\delta\neq 1$$
%Consider the following sets of points
%\begin{itemize}
%\item[A] $(z,1,0)$ with $\T(z)=0$ and $\T(\frac{z}{a})+\T(\frac{b}{a})=0$.
%\item[B]
%$(z,0,1)$ with  $\T(z)=0$ and $\T(\frac{z}{a\gamma})=0$
%\item[C] $(z,1,1)$ with $\T(z)=1$, $\T(\frac{z}{a\beta})=par$
%\item[D] $(z,\alpha+\delta,1)$ with $\T(\frac{z}{a\gamma})=0$ and $\T(\frac{z}{\alpha})=\T(\frac{b}{a})+1$
%\item[E] $(z,\delta,1)$ with $\T(\frac{z}{a\beta})=par+\T(\frac{b}{a})+1$ and $\T(\frac{z}{\alpha})=par$
%\end{itemize}
%Then the union of the sets $A,B,C,D,E$ is a $(0,2,q/4)$-arc of size $q+q/4$ in $\PG(2,q)$.
%\end{theorem}

\begin{remark}
  Recall that the construction of Theorem \ref{ourconstruction} depends on four parameters $\a,\beta\in\F_{q}\setminus\{0,1\}$ and $a,b\in\F_{2}$, $\a\beta\neq1$. Therefore, we denote this KM-arc of type $\frac{q}{4}$ by $\A_{\alpha,\beta,a,b}$.
\end{remark}

\begin{theorem}
  Let $\a,\beta\in\F_{q}\setminus\{0,1\}$, with $\a\beta\neq1$, and $a,a',b,b'\in\F_{2}$. The KM-arcs $\A_{\alpha,\beta,a,b}$ and $\A_{\alpha,\beta,a',b'}$ are $\PGL$-equivalent.
  %The example with $par=0$ is projectively equivalent to the one with $par=1$. ($par=c$) WAT MET $b$?
\end{theorem}
\begin{proof}
  Let $\nu\in\F_{q}$ be an element such that $\T(\nu)=0$ and $\T(\frac{\nu}{\alpha})=a+a'$. By Lemma \ref{tracesysteemalg} we know such an element exists since $\a\neq1$. Denote $\T(\frac{\nu}{\alpha\beta})$ by $k$.
  \par Let $\gamma$ be as in the construction of $\A_{\alpha,\beta,a,b}$ in Theorem \ref{ourconstruction}. Now, let $\rho$ be an element in $\F_{q}$ such that $\T(\rho)=0$, $\T(\frac{\rho}{\alpha\gamma})=0$ and $\T(\frac{\rho}{\alpha\beta})=k+b'+b$. Such an element exists by Lemma \ref{tracesysteemalg} since
  \[
    1+\frac{1}{\alpha\gamma}+\frac{1}{\alpha\beta}=\frac{\alpha\beta\gamma+\gamma+\beta}{\alpha\beta\gamma}=\frac{\alpha\beta\gamma+1+\alpha\beta\gamma}{\alpha\beta\gamma}=\frac{1}{\alpha\beta\gamma}\neq0\;.
  \]
  Now, the collineation induced by the matrix $\left(\begin{smallmatrix}1&\nu&\rho\\0&1&0\\0&0&1\end{smallmatrix}\right)$ maps $\A_{\alpha,\beta,a,b}$ on $\A_{\alpha,\beta,a',b'}$. To check that the sets $\S_{3}$ and $\S_{4}$ are indeed mapped onto points of $\A_{\alpha,\beta,a',b'}$ it is useful to note that
  \begin{align*}
    \T\left(\frac{\rho}{\xi}\right)&=\T\left(\rho\right)+\left(\frac{\rho}{\alpha\beta}\right)+\T\left(\frac{\rho}{\alpha\gamma}\right)=k+b+b'\\
    \T\left(\frac{\nu(\beta+1)}{\alpha\beta}\right)&=\T\left(\frac{\nu}{\alpha}\right)+\T\left(\frac{\nu}{\alpha\beta}\right)=k+a+a'\\
    \T\left(\frac{\nu(\beta+1)}{\xi}\right)&=\T\left(\frac{\nu(\alpha\beta+1)}{\alpha\beta}\right)=k\;.  \qedhere
  \end{align*}
\end{proof}

\begin{theorem}\label{transliff}
  Let $\a,\beta\in\F_{q}\setminus\{0,1\}$, with $\a\beta\neq1$. The KM-arc $\A_{\alpha,\beta,0,0}$ is a translation KM-arc if and only if $\a\in\{\frac{1}{\beta^{2}},1+\frac{1}{\beta},\beta,\frac{1}{\sqrt{\beta}},\frac{1}{\beta+1}\}$.
\end{theorem}
\begin{proof}
  First we prove that this condition on $\a$ and $\beta$ is necessary. We denote the nucleus $(1,0,0)$ of $\A_{\alpha,\beta,0,0}$ by $N$. Assume that the $\frac{q}{4}$-secant $\ell$ is a translation line of $\A_{\alpha,\beta,0,0}$. Denote the other four $\frac{q}{4}$-secants by $\ell_{1},\dots,\ell_{4}$. Let $G\leq\PGL(3,q)$ be the group containing all elations of $\PG(2,q)$ with axis $\ell$ that fix $\A_{\alpha,\beta,0,0}$. We consider the natural action of $G$ on the set $\mathcal{L}$ of lines through $N$. The elation $\varphi$ acts trivially on $\mathcal{L}$ if and only if $N$ is the center of $\varphi$. Let $H\leq G$ be the subgroup of elations with center $N$. If $N$ is not the center of $\varphi$, then $\varphi$ fixes precisely one element of $\mathcal{L}$, namely $\ell$. 
  \par The set $\mathcal{L}'=\{\ell_{1},\ell_{2},\ell_{3},\ell_{4}\}\subset\mathcal{L}$ is stabilised by the elements of $G$ since $G$ maps the $\frac{q}{4}$-secants necessarily onto each other. We know that an element $\varphi\in G\setminus H$ cannot have order $3$, for if $\varphi^{3}=1$, then the unique line in $\mathcal{L}'\setminus\{\ell_{1},\ell^{\varphi}_{1},\ell^{\varphi^{2}}_{1}\}$ has to be fixed, a contradiction. Consequently, $\varphi^{4}$ acts trivially on $\mathcal{L}$ for all $\varphi\in G$.
  \par Now, we assume that $\ell$ is the line $Z=0$. Let $Y=x_{i}Z$ be the equation of the line $\ell_{i}$. Any elation $\varphi$ with axis $\ell$ is given by a matrix $A=\left(\begin{smallmatrix}1&0&u\\0&1&v\\0&0&w\end{smallmatrix}\right)$, with $u,v,w\in\F_{q}$. If $\varphi\in G$, then $w^{4}=1\Leftrightarrow w=1$ since $\ell^{\varphi^{4}}_{i}=\ell_{i}$ for all $i=1,\dots,4$. We can find a $\varphi\in G$ such that $\ell^{\varphi}_{1}=\ell_{2}$ since $\ell$ is a translation line. We find that $x_{2}=x_{1}+v$. Moreover, it also follows that $\ell^{\varphi}_{3}=\ell_{4}$ (also $\ell^{\varphi}_{2}=\ell_{1}$ and $\ell^{\varphi}_{4}=\ell_{3}$). So, $x_{4}=x_{3}+v$. Consequently, $x_{1}+x_{2}+x_{3}+x_{4}=0$. Applying this to $\A_{\alpha,\beta,0,0}$, we find that if $Z=0$ is a translation line then
  \[
    0+1+\gamma+(1+\beta)=0\quad\Leftrightarrow\quad \beta=\gamma\quad\Leftrightarrow\quad \a=\frac{1}{\beta^{2}}\;.
  \]
  Using a collineation in $\PGL(3,q)$ that stabilises $N$ and that maps $\ell_{i}$ onto the line $Z=0$, we can find the necessary conditions for the other $\frac{q}{4}$-secants to be the translation line. The lines $\ell_{1}$, $\ell_{2}$, $\ell_{3}$ and $\ell_{4}$ are translation lines if
  \begin{align*}
    0+1+\frac{1}{\gamma}+\frac{1}{1+\beta}=0\quad &\Leftrightarrow\quad \gamma=1+\frac{1}{\beta}&&\Leftrightarrow\quad \a=1+\frac{1}{\beta}\\
    1+0+\frac{\gamma}{\gamma+1}+\frac{1+\beta}{\beta}=0\quad &\Leftrightarrow\quad \gamma=\frac{1}{\beta+1}&&\Leftrightarrow\quad \a=\beta\\
    \frac{1}{\gamma}+\frac{1}{\gamma+1}+\frac{1}{\beta+\gamma+1}+0=0\quad&\Leftrightarrow\quad\gamma=\sqrt{\beta+1}&&\Leftrightarrow\quad \a=\frac{1}{\sqrt{\beta}}\\
    0+\frac{1}{\beta+1}+\frac{1}{\beta}+\frac{1}{\beta+\gamma+1}=0\quad&\Leftrightarrow\quad\gamma=\beta^{2}+1&&\Leftrightarrow\quad \a=\frac{1}{\beta+1}\;,
  \end{align*}
  respectively. This concludes the first part of the proof.
  \par Now, we need to prove that the condition $\a\in\{\frac{1}{\beta^{2}},1+\frac{1}{\beta},\beta,\frac{1}{\sqrt{\beta}},\frac{1}{\beta+1}\}$ is sufficient for $\A_{\alpha,\beta,0,0}$ to be a translation KM-arc. We show that $\ell:Z=0$ is a translation line if $\a=\frac{1}{\beta^{2}}$. We consider the point $P(0,0,1)\in\A_{\alpha,\beta,0,0}\setminus\ell$. It is sufficient to prove that for any point $R\in\A_{\alpha,\beta,0,0}\setminus\ell$ the unique elation with axis $\ell$ mapping $P$ onto $R$ fixes $\A_{\alpha,\beta,0,0}$.
  \par Any point in $\A_{\alpha,\beta,0,0}\setminus\ell$ can be written in one of the following ways: 
  \begin{align*}
    &(\mu_{1},0,1)\text{, with } \T(\mu_{1})=0,\T(\beta\mu_{1})=0\;,&(\mu_{2},1,1)\text{, with } \T(\mu_{2})=1,\T(\beta\mu_{2})=0\;,\\
    &(\mu_{3},\beta,1)\text{, with } \T(\mu_{3})=1,\T(\beta\mu_{3})=1\;,&(\mu_{4},\beta+1,1)\text{, with } \T(\mu_{4})=0,\T(\beta\mu_{4})=1\;.
  \end{align*}
  It can be calculated that the elations with axis $\ell$ given by the matrices
  \[
    \begin{pmatrix}
      1&0&\mu_{1}\\
      0&1&0\\
      0&0&1
    \end{pmatrix}\;,
    \begin{pmatrix}
      1&0&\mu_{2}\\
      0&1&1\\
      0&0&1
    \end{pmatrix}\;,
    \begin{pmatrix}
      1&0&\mu_{3}\\
      0&1&\beta\\
      0&0&1
    \end{pmatrix}\text{ and }
    \begin{pmatrix}
      1&0&\mu_{4}\\
      0&1&\beta+1\\
      0&0&1
    \end{pmatrix}
  \]
  fix the KM-arc $\A_{\alpha,\beta,0,0}$, and map $P$ onto $(\mu_{1},0,1)$, $(\mu_{2},1,1)$, $(\mu_{3},\beta,1)$ and $(\mu_{4},\beta+1,1)$, respectively. Consequently, $Z=0$ is a translation line.
  \par The arguments for the other $\frac{q}{4}$-secants are analogous.
  %\par VOORWAARDES NOG EENS GOED NAKIJKEN BIJ OPSCHRIJVEN BEWIJS
\end{proof}

\begin{remark}
  For a given $\beta\in\F_{q}\setminus\{0,1\}$, $q$ even, we consider the set $\Theta=\{\frac{1}{\beta^{2}},1+\frac{1}{\beta},\beta,\frac{1}{\sqrt{\beta}},\frac{1}{\beta+1}\}$. If two of the elements in $\Theta$ coincide, then they all coincide, and $\Theta$ contains precisely one element. In this case, $\beta^{3}=1$. So, $\a=\beta$ is necessarily contained in the subfield $\F_{4}\subseteq\F_{q}$, and hence $q$ must be a square. So, if $q$ is a square, and $\beta$ is contained in $\F_{4}\subseteq\F_{q}$, then there is only one value for $\a$ such that $\A_{\alpha,\beta,0,0}$ is a translation arc, but in this case it is translation with respect to all of its $\frac{q}{4}$-secants. If $q$ is a square, but $\beta$ is not contained in $\F_{4}\subseteq\F_{q}$, or if $q$ is not a square, then there are precisely five values for $\a$ such that $\A_{\alpha,\beta,0,0}$ is a translation arc. In this case it is a translation arc with respect to only one of its $\frac{q}{4}$-secants.
\end{remark}

\begin{remark} It is not hard to check that $\A_{\frac{1}{\beta^{2}},\beta,0,0}\cong\A_{\frac{1}{\beta^{2}+1},\beta+1,0,0}$, so by choosing different parameters $\alpha,\beta$, we may end up with $\PGL$-equivalent KM-arcs of the form $\A_{\alpha,\beta,0,0}$. However, as follows from Theorem \ref{transliff}, not all KM-arcs of the form $\A_{\alpha,\beta,0,0}$ are translation KM-arcs, so the family $\A_{\a,\beta,0,0}$ where $\a$ and $\beta$ vary, certainly contains KM-arcs which are not $\PGL$-equivalent.
\end{remark}

In Theorems \ref{case1} and \ref{case2}, we will prove that KM-arcs that satisfy property (I) or property (II) arise from the construction of Theorem \ref{ourconstruction}.
Before we give the proof, we intruduce some notation to make a distinction between points of $\PG(2,2^h)$ and points of $\PG(3h-1,2)$. We will denote a point of the latter space by a vector determining this point, using the subscript $\F_2$. So for a fixed $x\in \F_{q^h}^\ast$, $(x,\T(x),0)_{\F_2}$ is a point of $\PG(3h-1,2)$. Note that $\{(x,\T(x),0)_{\F_2}\mid x\in\F_{2^h}^\ast\}$ and $\{(\zeta x,\zeta\T(x),0)_{\F_2}\mid x\in\F_{2^h}^\ast\}$, with $\zeta\in\F_{q}\setminus\{0,1\}$ are different projective subspaces of $\PG(3h-1,2)$, but they determine the same (linear) point set in $\PG(2,2^h)$.

%We will present the subspaces in $\PG(3h-1,2)$ by their corresponding linear sets in $\PG(2,q)$, using a subscript $_{2}$ to indicate that they should be considered in $\PG(3h-1,2)$, e.g. $\{(x,\T(x),0)_{2}\mid x\in\F_{q}\}$ and $\{(\zeta x,\zeta\T(x),0)_{2}\mid x\in\F_{q}\}$, with $\zeta\in\F_{q}\setminus\{0,1\}$ are different subspaces in $\PG(3h-1,2)$, but they give rise to the same linear set in $\PG(2,q)$.

\begin{theorem}\label{case1}
  Let $\A$ be a KM-arc of type $\frac{q}{4}$ in $\PG(2,q)$, with $q=2^{h}$. If $\ell$ is a $\frac{q}{4}$-secant of $\A$ such that $\A$ has property (II) with respect to $\ell$, then $\A$ is $\PGL$-equivalent to an arc given in Theorem \ref{ourconstruction} with $\a\beta^{2}\neq1$.
\end{theorem}
\begin{proof}
  Recall that $\T$ is the absolute trace function from $\F_{q}$ to $\F_{2}$. Let $N$ be the nucleus of $\A$. Denote the four $\frac{q}{4}$-secants of $\A$ different from $\ell$ by $\ell_{1},\ell_{2},\ell_{3},\ell_{4}$, and denote $\A\cap\ell_{i}$ by $\A_{i}$ and $\A_{i}\cup\{N\}$ by $\A^{N}_{i}$, $i=1,\dots,4$. We also use the notations
  \[
    \D_{ij}=\{\left\langle P,Q\right\rangle\cap\ell\mid P\in\A_{i},Q\in\A_{j}\}\text{ and }\D^{N}_{ij}=\D_{ij}\cup\{N\}\qquad 1\leq i<j\leq 4\;.
  \]
  %First, we assume that $\A$ has property (II) with respect to $\ell$. 
  Since $\A$ has property (II) we know that $\D_{12}=\D_{34}$ and that $\D^{N}_{12}$ is an $(h-1)$-club of rank $h$ with head $N$ in $\ell$. By Theorem \ref{clueq} we can choose a frame of $\PG(2,q)$ such that $\D^{N}_{12}$ is the set $\{(x,\T(x),0)\mid x\in\F_{q}^\ast\}$, such that $(0,0,1)\in\A_{1}$ and $(1,1,1)\in\ell_{2}$. It follows that $(1,0,0)$ is the nucleus $N$ of $\A$ and that $\D_{12}=\{(x,1,0)\mid x\in\F_{q},\T(x)=1\}$. As all $(h-1)$-clubs of rank $h$ are $\PGL$-equivalent and as $\D^{N}_{13}$ has head $N$, there exists an $\a\in\F_{q}\setminus\{0,1\}$ and an $\a'\in\F_{q}$ such that $\D^{N}_{13}=\{(\a x+\a'\T(x),\T(x),0)\mid x\in\F_{q}^\ast\}$. Denote $\T\left(\frac{\a'}{\alpha}\right)=a$. It follows that the set $\D_{13}=\D_{24}$ equals $\{(\a x+\a',1,0)\mid x\in\F_{q},\T(x)=1\}$. Now, $\D^{N}_{14}$ is the unique $(h-1)$-club of rank $h$ with head $N$ that has $\frac{q}{4}$ points in common with both $\D_{12}$ and $\D_{13}$, but none with $\D_{12}\cap\D_{13}$. Hence, $\D^{N}_{14}$ equals $\{\left(\frac{\alpha}{\alpha+1}x+\frac{\a'}{\alpha+1}\T(x),\T(x),0\right)\mid x\in\F_{q}^\ast\}$ and consequently $\D_{14}=\D_{23}$ equals $\{\left(\frac{\alpha}{\alpha+1}x+\frac{\a'}{\alpha+1},1,0\right)\mid x\in\F_{q},\T(x)=1\}$.
  \par Let $\D$ be a Desarguesian spread in $\PG(3h-1,2)$, let $H$ be the $(2h-1)$-space such that $\ell=\B(H)$, and let $\rho\in\D$ be the spread element such that $\B(\rho)=N$. 
  \par We define $\sigma$ as the $(2h)$-space in $\PG(3h-1,2)$ spanned by $H$ and $(0,0,1)_{\F_2}$. Then, there is a unique $(h-2)$-space $ \sigma_{i}\subset\sigma$ such that $\B(\sigma_{i})=\A^{N}_{i}$, $i=1,\dots,4$. For every $i=1,\ldots,4$, the subspace $\sigma_{i}$ meets $\rho$ in an $(h-3)$-space; the existence of these subspaces $\sigma_{i}$ follows from the first condition for KM-arcs of property (II). Let $\mu_{ij}$ be the subspace $H\cap\left\langle\sigma_{i},\sigma_{j}\right\rangle$, for $1\leq i<j\leq 4$. Then, $\D^{N}_{ij}=\B(\mu_{ij})$, for $1\leq i<j\leq 4$.
  \par The Desarguesian spread $\D$ is fixed, and we have that $\B(\mu_{12})=\{(x,\T(x),0)\mid x\in \F_{q}^\ast\}$. This leaves us the freedom to choose coordinates of $H=\PG(2h-1,2)$ in such a way that $\mu_{12}=\{(x,\T(x),0)_{\F_2}\mid x\in\F_{q}^\ast\}$. Now we can find $\xi,\beta,\gamma,\delta,\varepsilon\in\F^{*}_{q}$ such that
  \begin{align*}
    %\mu_{12}&=\{(x,\T(x),0)_{2}\mid x\in\F_{q}\}\\
    \mu_{34}&=\{(\xi x,\xi\T(x),0)_{\F_2}\mid x\in\F_{q}^\ast\}\\
    \mu_{13}&=\{(\gamma(\alpha x+\alpha'\T(x)),\gamma\T(x),0)_{\F_2}\mid x\in\F_{q}^\ast\}\\
    \mu_{24}&=\{(\beta(\alpha x+\alpha'\T(x)),\beta\T(x),0)_{\F_2}\mid x\in\F_{q}^\ast\}\\
    \mu_{14}&=\left\{\left(\delta\left(\frac{\alpha}{\alpha+1}x+\frac{\alpha'}{\alpha+1}\T(x)\right),\delta\T(x),0\right)_{\F_2}\mid x\in\F_{q}^\ast\right\}\\
    \mu_{23}&=\left\{\left(\varepsilon\left(\frac{\alpha}{\alpha+1}x+\frac{\alpha'}{\alpha+1}\T(x)\right),\varepsilon\T(x),0\right)_{\F_2}\mid x\in\F_{q}^\ast\right\}\;.
  \end{align*}
  The $(h-2)$-space $\sigma_{1}$ must necessarily meet $\rho$ in the $(h-3)$-space $\tau_{1}=\mu_{12}\cap\mu_{13}=\{(x,0,0)\mid x\in \F_{q}^\ast, \T(x)=\T(\frac{x}{\alpha\gamma})=0\}$. By Lemma \ref{tracesysteemalg} it follows that $\a\gamma\neq1$. We know that $\mu_{14}$ must pass through $\mu_{12}\cap\mu_{13}$. Again by Lemma \ref{tracesysteemalg}, we find the necessary condition that
  \[
    1+\frac{1}{\alpha\gamma}+\frac{\alpha+1}{\alpha\delta}=0\quad\Leftrightarrow\quad\delta=\frac{(\alpha+1)\gamma}{\alpha\gamma+1}\;.
  \]
  Analogously, we know that $\tau_{2}=\mu_{12}\cap\mu_{23}\cap\mu_{24}$, $\tau_{3}=\mu_{13}\cap\mu_{23}\cap\mu_{34}$ and $\tau_{4}=\mu_{14}\cap\mu_{24}\cap\mu_{34}$ are $(h-3)$-spaces in $\rho$. It follows that $\a\beta\neq1$, $\gamma(\alpha+1)\neq\varepsilon$ and $\beta(\alpha+1)\neq\delta$, and moreover
  \begin{align*}
    1+\frac{1}{\alpha\beta}+\frac{\alpha+1}{\alpha\varepsilon}=0\quad&\Leftrightarrow\quad\varepsilon=\frac{(\alpha+1)\beta}{\alpha\beta+1}\;,\\
    \frac{1}{\alpha\gamma}+\frac{\alpha+1}{\alpha\varepsilon}+\frac{1}{\xi}=0\quad&\Leftrightarrow\quad\frac{1}{\beta}+\frac{1}{\gamma}=\a\left(1+\frac{1}{\xi}\right)\;.
  \end{align*}
  Note that the equation $\frac{\alpha+1}{\alpha\delta}+\frac{1}{\alpha\beta}+\frac{1}{\xi}=0$ derived from $\tau_{4}$ is redundant.
  \par If the $(h-3)$-spaces $\tau_{i}$ and $\tau_{j}$ coincide, then $\left\langle\sigma_{i},\sigma_{j}\right\rangle\cap H$ cannot be an $(h-1)$-space. Hence, no two of the $(h-3)$-spaces $\tau_{1},\dots,\tau_{4}$ may coincide. Consequently, $\beta\neq\gamma$ and $\xi\neq1$.
  \par Since the set $\A^{N}_{1}$ contains $(0,0,1)$ and since $\sigma_{1}\subset\sigma$, the subspace $\sigma_{1}$ must contain $(0,0,1)_{\F_2}$. Hence, the points of $\sigma_{1}\setminus H$ are the points $\{(x,0,1)_{\F_2}\mid x\in \F_q, \T(x)=\T(\frac{x}{\alpha\gamma})=0\}$ and consequently $\A_{1}=\{(x,0,1)\mid x\in \F_q, \T(x)=\T(\frac{x}{\alpha\gamma})=0\}$.
  \par Let $\rho'\in\D$ be the spread element such that $\B(\rho')$ is the point $(1,1,1)$. Then the intersection $\rho'\cap\sigma$ only contains one point, namely $(1,1,1)_{\F_{2}}$. The $(h-2)$-space $\sigma_{2}$ is contained in $\left\langle \rho,\rho'\right\rangle$ since $(1,1,1)\in\ell_{2}$. Hence, it is contained in the $h$-space $\left\langle\rho,(1,1,1)_{\F_2}\right\rangle=\left\langle\rho,(0,1,1)_{\F_2}\right\rangle$. Furthermore, the $(h-2)$-space $\sigma_{2}$ passes through the $(h-3)$-space $\mu_{12}\cap\mu_{24}=\{(x,0,0)_{\F_2}\mid x\in \F_{q}^\ast, \T(x)=\T(\frac{x}{\alpha\beta})=0\}\subset\rho$. So, there is a $t'\in\F_{q}$ such that $\sigma_{2}=\left\langle(t',1,1)_{\F_2},\mu_{12}\cap\mu_{24}\right\rangle$. Note that $\T(t')=1$ since $\left\langle(t',1,1),(0,0,1)\right\rangle\cap\ell=(t',1,0)$ has to be a point of $\D_{12}$. Denote $\T(\frac{t'}{\alpha\beta})$ by $b$, then, $\A_{2}$ is given by $\{(x,1,1)\mid x\in \F_q, \T(x)=1,\T(\frac{x}{\alpha\beta})=b\}$.
  %\par The $(h-2)$-space $\sigma_{2}$ passes through the $(h-3)$-space $\mu_{12}\cap\mu_{24}=\{(x,0,0)_{\F_2}\mid x\in \F_{q}^\ast, \T(x)=\T(\frac{x}{\alpha\beta})=0\}\subset\rho$, and is contained in the $h$-space $\left\langle\rho,(0,1,1)_{\F_2}\right\rangle$ since $(1,1,1)\in\ell_{2}$. It is necessarily also contained in $\left\langle\mu_{12},(0,0,1)_{\F_2}\right\rangle$. Hence, there is a $t'\in\F_{q}$ with $\T(t')=1$ such that $\sigma_{2}=\left\langle(t',1,1)_{\F_2},\mu_{12}\cap\mu_{24}\right\rangle$. Denote $\T(\frac{t'}{\alpha\beta})$ by $b$. Then, $\A_{2}$ is given by $\{(x,1,1)\mid x\in \F_q, \T(x)=1,\T(\frac{x}{\alpha\beta})=b\}$.
  \par The $(h-2)$-space $\sigma_{3}$ passes through the $(h-3)$-space $\mu_{13}\cap\mu_{34}=\{(x,0,0)_{\F_2}\mid x\in \F_{q}^\ast, \T(\frac{x}{\xi})=\T(\frac{x}{\alpha\gamma})=0\}\subset\rho$, and is contained in the $h$-spaces $\left\langle\mu_{13},(0,0,1)_{\F_2}\right\rangle$ and $\left\langle\mu_{23},(t',1,1)_{\F_2}\right\rangle$. So, the points of $\sigma_{3}\setminus H$ are the points given by the following coordinates for both an $x$ and a $y$ in $\F_{q}$:
  \[
    \left(\gamma \alpha x+\gamma \alpha'\T(x),\gamma\T(x),1\right)_{\F_2}=\left(\frac{\varepsilon\alpha}{\alpha+1}y+\frac{\varepsilon \alpha'}{\alpha+1}\T(y)+t',\varepsilon\T(y)+1,1\right)_{\F_2}\;.
  \]
  Since $\A_{3}$ is not on the lines $\ell_{1}$ or $\ell_{2}$, necessarily $\gamma\T(x)=\varepsilon\T(y)+1\notin\{0,1\}$. Thus $\T(x)=\T(y)=1$, and moreover also $\gamma=\varepsilon+1$. Combining this with previous results $\varepsilon=\frac{(\alpha+1)\beta}{\alpha\beta+1}$, $\delta=\frac{(\alpha+1)\gamma}{\alpha\gamma+1}$ and $\frac{1}{\xi}=1+\frac{1}{\alpha\beta}+\frac{1}{\alpha\gamma}$, we find
  \[
    \gamma=\frac{\beta+1}{\alpha\beta+1}\quad\text{and also}\quad\xi=\a\beta\gamma\quad\text{and}\quad\delta=\beta+1\;.
  \]
  Now, as $\sigma_{3}$ passes through $\mu_{13}\cap\mu_{34}$ we know that the points of $\sigma_{3}\setminus H$ are given by
  \begin{align*}
    &\left\{(z,\gamma,1)_{\F_2}\mid z \in \F_q, \T\left(\frac{z}{\xi}\right)=d,\T\left(\frac{z}{\alpha\gamma}\right)=d'\right\}\\&=\left\{(z,\gamma,1)_{\F_2}\mid z\in \F_q, \T\left(\frac{z}{\xi}\right)=d,\T\left(\frac{z}{\xi}\right)+\T\left(\frac{z}{\alpha\gamma}\right)=d+d'\right\}
  \end{align*}
  for some $d,d'\in\F_{2}$. We find that
  \begin{align*}
    d'&=\T\left(\frac{\gamma \a x+\gamma \alpha'\T(x)}{\alpha\gamma}\right)=\T(x)+\T\left(\frac{\alpha'}{\alpha}\right)=a+1\\
    d+d'&=\T\left(\frac{\alpha+1}{\alpha\varepsilon}\left(\frac{\alpha\varepsilon y+\varepsilon \alpha'}{\alpha+1}+t'\right)\right)=\T(y)+\T\left(\frac{\alpha'}{\alpha}\right)+\T(t')+\T\left(\frac{t'}{\alpha\beta}\right)\\
    &=1+a+1+b=a+b\;.
  \end{align*}
  Hence, $d=b+1$ and the points of $\A_{3}$ are given by $\{(x,\gamma,1)\mid x\in \F_q,\T(\frac{x}{\alpha\gamma})=a+1,\T(\frac{x}{\xi})=b+1\}$.
  \par The $(h-2)$-space $\sigma_{4}$ passes through $\mu_{24}\cap\mu_{34}=\{(x,0,0)_{\F_2}\mid x\in \F_q^\ast, \T(\frac{x}{\xi})=\T(\frac{x}{\alpha\beta})=0\}\subset\rho$ and is contained in the $h$-spaces $\left\langle\mu_{14},(0,0,1)_{\F_2}\right\rangle$ and $\left\langle\mu_{24},(t',1,1)_{\F_2}\right\rangle$. We proceed in the same way as for $\sigma_{3}$. The points of $\sigma_{4}\setminus H$ are the points given by the following coordinates for both an $x$ and a $y$ in $\F_{q}$:
  \[
    \left(\frac{\delta\alpha}{\alpha+1}x+\frac{\delta\alpha'}{\alpha+1}\T(x),(\beta+1)\T(x),1\right)_{\F_2}=\left(\beta \alpha y+\beta \alpha'\T(y)+t',\beta\T(y)+1,1\right)_{\F_2}\;.
  \]
  Since $(\beta+1)\T(x)=\beta\T(y)+1\notin\{0,1\}$, necessarily $\T(x)=\T(y)=1$. Now, as $\sigma_{4}$ passes through $\mu_{24}\cap\mu_{34}$ we know that the points of $\sigma_{4}\setminus H$ are given by
  \begin{align*}
    &\left\{(z,\beta+1,1)_{\F_2}\mid z \in \F_q, \T\left(\frac{z}{\xi}\right)=e,\T\left(\frac{z}{\alpha\beta}\right)=e'\right\}\\&=\left\{(z,\beta+1,1)_{\F_2}\mid z\in \F_q, \T\left(\frac{z}{\alpha\beta}\right)=e',\T\left(\frac{z}{\xi}\right)+\T\left(\frac{z}{\alpha\beta}\right)=e+e'\right\}
  \end{align*}
  for some $e,e'\in\F_{2}$. We find that
  \begin{align*}
    e'&=\T\left(\frac{\beta \alpha y+\beta \alpha'\T(y)+t'}{\alpha\beta}\right)=\T(y)+\T\left(\frac{\alpha'}{\alpha}\right)+\T\left(\frac{t'}{\alpha\beta}\right)=1+a+b\\
    e+e'&=\T\left(\frac{\alpha+1}{\alpha\delta}\left(\frac{\alpha\delta x+\delta \alpha'}{\alpha+1}\right)\right)=\T(x)+\T\left(\frac{\alpha'}{\alpha}\right)=1+a\;.
  \end{align*}  
  We conclude that the points of $\A_{4}$ are given by $\{(x,\beta+1,1)\mid x\in \F_q,\T(\frac{x}{\xi})=b,\T(\frac{x}{\alpha\beta})=a+b+1\}$.
  \par The points of $\A\cap\ell$ are the $\frac{q}{4}$ points of $\ell\setminus\{N\}$ that are not contained in $\D_{12}\cup\D_{13}\cup\D_{14}$. Recall that $\D_{12}=\{(x,1,0)\mid\T(x)=1\}$, $\D_{13}=\{(x,1,0)\mid x\in \F_q, \T(\frac{x}{\alpha})=a+1\}$ and $\D_{14}=\{(x,1,0)\mid x\in\F_q,\T(x)+\T(\frac{x}{\alpha})=a+1\}$. So the points of $\A\cap\ell$ are the points $\{(x,1,0)\mid x\in \F_q, \T(x)=0,\T(\frac{x}{\alpha})=a\}$.
  \par Note that the parameters $\xi$, $\gamma$, $\delta$ and $\varepsilon$ are written in function of $\a$ and $\beta$ in the following way:
  \[
    \gamma=\frac{\beta+1}{\alpha\beta+1},\quad\delta=\beta+1,\quad\varepsilon=\frac{(\alpha+1)\beta}{\alpha\beta+1},\quad\xi=\alpha\beta\gamma\;.
  \]
  The conditions on these parameters that we found during the proof, namely
  \[
    \a,\xi,\beta,\gamma,\delta,\varepsilon\neq0\qquad \a\beta,\a\gamma\neq1\qquad \varepsilon\neq(\alpha+1)\gamma,\;\delta\neq(\alpha+1)\beta \qquad\beta\neq\gamma,\;\xi\neq1\;,
  \]
  can thus be rewritten as conditions on $\a$ and $\beta$. We find that $\a,\beta\notin\{0,1\}$, $\a\beta\neq1$ and $\a\beta^{2}\neq1$. So, indeed $\A$ is $\PGL$-equivalent to a KM-arc given in Theorem \ref{ourconstruction}. In fact, all examples in Theorem \ref{ourconstruction} are described except the ones with $\a\beta^{2}=1$. This finishes the proof.% in case $\A$ has property (II).
\end{proof}

\begin{theorem}\label{case2}
  Let $\A$ be a KM-arc of type $\frac{q}{4}$ in $\PG(2,q)$, with $q=2^{h}$. If $\ell$ is a $\frac{q}{4}$-secant of $\A$ such that $\A$ has property (I) with respect to $\ell$, then $\A$ is $\PGL$-equivalent to an arc given in Theorem \ref{ourconstruction} with $\a\beta^{2}=1$, and hence, is a translation arc with translation line $\ell$.
\end{theorem}
\begin{proof}
  Recall that $\T$ is the absolute trace function from $\F_{q}$ to $\F_{2}$. Let $N$ be the nucleus of $\A$. Denote the four $\frac{q}{4}$-secants of $\A$ different from $\ell$ by $\ell_{1},\ell_{2},\ell_{3},\ell_{4}$, and denote $\A\cap\ell_{i}$ by $\A_{i}$ and $\A_{i}\cup\{N\}$ by $\A^{N}_{i}$, $i=1,\dots,4$. We also use the notations
  \[
    \D_{ij}=\{\left\langle P,Q\right\rangle\cap\ell\mid P\in\A_{i},Q\in\A_{j}\}\text{ and }\D^{N}_{ij}=\D_{ij}\cup\{N\}\qquad 1\leq i<j\leq 4\;.
  \]
  Since $\A$ has property (I) we know that $\D_{12}=\D_{34}$ and that $\D^{N}_{12}$ is an $(h-2)$-club of rank $h-1$ with head $N$ in $\ell$. Then we can find two $(h-1)$-clubs of rank $h$ with head $N$, say $\C$ and $\C'$, such that $\D^{N}_{12}=\C\cap\C'$. By Theorem \ref{clueq} we can choose a frame of $\PG(2,q)$ such that $\C$ is the set $\{(x,\T(x),0)\mid x\in\F^{*}_{q}\}$ and such that $(0,0,1)\in\A_{1}$. It follows that $N=(1,0,0)$. By Theorem \ref{clueq} all $(h-1)$-clubs of rank $h$ are equivalent, so $\C'$ is given by $\{(\frac{x}{\beta}+\beta'\T(x),\T(x),0)\mid x\in\F^{*}_{q}\}$ for some $\beta\in\F_{q}\setminus\{0,1\}$ and some $\beta'\in\F_{q}$. Denote $\T(\beta \beta')=b$. Hence, $\D_{12}$ is given by $\{(x,1,0)\mid x\in \F_q, \T(x)=1,\T(\beta x)=b\}$.
  \par Now, let $\D$ be a Desarguesian spread in $\PG(3h-1,2)$, let $H$ be the $(2h-1)$-space such that $\ell=\B(H)$, and let $\rho\in\D$ be the spread element such that $\B(\rho)=N$.
  \par We define $\sigma$ as the $(2h)$-space in $\PG(3h-1,2)$ spanned by $H$ and $(0,0,1)_{\F_2}$. Then there is a unique $(h-2)$-space $ \sigma_{i}\subset\sigma$ such that $\B(\sigma_{i})=\A^{N}_{i}$, $i=1,\dots,4$. For every $i=1,\ldots,4$, the subspace $\sigma_{i}$ meets $\rho$ in an $(h-3)$-space; the existence of these subspaces $\sigma_{i}$ follows from the first condition for KM-arcs of property (I). Let $\mu_{ij}$ be the subspace $H\cap\left\langle\sigma_{i},\sigma_{j}\right\rangle$, for $1\leq i<j\leq 4$. Then, $\D^{N}_{ij}=\B(\mu_{ij})$, for $1\leq i<j\leq 4$ and $\mu_{ij}\subset H$ is an $(h-2)$-space meeting $\rho$ in an $(h-3)$ space, $1\leq i<j\leq4$. However, $\mu_{ij}\cap\rho$ must contain the $(h-3)$-spaces $\sigma_{i}\cap\rho$ and $\sigma_{j}\cap\rho$, so necessarily $\mu_{ij}\cap\rho=\sigma_{i}\cap\rho=\sigma_{j}\cap\rho$. We find that there is an $(h-3)$-space $\tau\subset\rho$ such that $\sigma_{i}\cap\rho=\tau$ for $i=1,\dots,4$, and such that $\tau=\mu_{ij}\cap\rho$ for $1\leq i<j\leq4$.
  \par Let $\mu$ and $\mu'$ be the two $(h-1)$-spaces in $H$ such that $\mu_{12}=\mu\cap\mu'$ and such that $\B(\mu)=\C$ and $\B(\mu')=\C'$. Now, after having fixed the spread $\D$, we still can choose coordinates for $H$ in such a way that $\mu=\{(x,\T(x),0)_{\F_2}\mid x\in\F^{*}_{q}\}$. Then, the space $\mu'$ is given by $\{(\frac{x}{\beta}+\beta'\T(x),\T(x),0)_{\F_2}\mid x\in\F^{*}_{q}\}$. Hence, the points of $\mu_{12}\setminus\tau$ are given by $\{(x,1,0)_{\F_2}\mid x\in \F_q,\T(x)=1,\T(\beta x)=b\}$, and $\tau$ is given by $\{(x,0,0)_{\F_2}\mid x\in \F^{*}_q,\T(x)=\T(\beta x)=0\}$.
  \par The $(h-2)$-space $\sigma_{1}$ is the space $\left\langle\tau,(0,0,1)_{\F_2}\right\rangle$, and hence the points of $\A_{1}$ are given by $\{(x,0,1)\mid\T(x)=\T(\beta x)=0\}$. The subspace $\sigma_{2}$ is the unique $(h-2)$-space through $\tau$ in the $(h-1)$-space $\left\langle\mu_{12},\sigma_{1}\right\rangle$, different from $\mu_{12}$ and $\sigma_{1}$. It follows immediately that $\A_{2}=\{(x,1,1)\mid x\in \F_q,\T(x)=1,\T(\beta x)=b\}$.
  \par For all $(h-2)$-spaces in $H$ meeting $\rho$ in the $(h-3)$-space $\tau$ the points of this space outside $\tau$ are given by $\{(x,\zeta,0)_{\F_2}\mid x\in \F_q, \T(x)=e,\T(\beta x)=e'\}$ for a $\zeta\in\F^{*}_{q}$ and $e,e'\in\F_{2}$. %can be described as $\{(x,\zeta(\T(x)+e+1)(\T(\beta x)+e'+1),0)_{2}\mid x\in\F^{*}_{q}\}$ for a $\zeta\in\F^{*}_{q}$ and $e,e'\in\F_{2}$. Alternatively, we can say that 
  So, there exist $\eta,\theta\in\F^{*}_{q}$ and $d_{3},d'_{3},d_{4},d'_{4}\in\F_{2}$ such that $\mu_{13}=\{(x,\eta,0)_{\F_2}\mid x\in \F_q, \T(x)=d_{3},\T(\beta x)=d'_{3}\}$ and $\mu_{14}=\{(x,\theta,0)_{\F_2}\mid x\in \F_q,\T(x)=d_{4},\T(\beta x)=d'_{4}\}$. It follows that the points of $\sigma_{3}\setminus\tau$ are the points $\{(x,\eta,1)_{\F_2}\mid x\in \F_q,\T(x)=d_{3},\T(\beta x)=d'_{3}\}$ and thus $\A_{3}=\{(x,\eta,1)\mid x\in \F_q,\T(x)=d_{3},\T(\beta x)=d'_{3}\}$. Analogously, $\sigma_{4}\setminus\tau=\{(x,\theta,1)_{\F_2}\mid x\in \F_q,\T(x)=d_{4},\T(\beta x)=d'_{4}\}$ and $\A_{4}=\{(x,\theta,1)\mid x\in \F_q,\T(x)=d_{4},\T(\beta x)=d'_{4}\}$. Necessarily, $1\neq\eta\neq\theta\neq1$.
  \par We consider a line passing through an arbitrary point of $\A_{1}$ and an arbitrary point of $\A_{2}$, the line $\ell_{12}:\left\langle(x,0,1),(y,1,1)\right\rangle$ with $\T(x)=0=\T(\beta x)$ and $\T(y)=1,\T(\beta y)=b$. This line cannot contain a point of $\A_{3}$ since $\A$ is a KM-arc. Hence, the point $((\eta+1)x+\eta y,\eta,1)\in\ell_{12}$ is not a point of $\A_{3}$. So, $\T(\eta(x+y))\neq d_{3}$ or $\T(\eta\beta(x+y))\neq d'_{3}$. Analogously, considering $\A_{4}$ we find that $\T(\theta(x+y))\neq d_{4}$ or $\T(\theta\beta(x+y))\neq d'_{4}$ Consequently, in general the systems of equations
  \begin{align}\label{eq1}
    \begin{cases}
      \T(z)=1\\
      \T(\beta z)=b\\
      \T(\eta z)=d_{3}\\
      \T(\eta \beta z)=d'_{3}
    \end{cases}\quad\text{ and }\quad
    \begin{cases}
      \T(z)=1\\
      \T(\beta z)=b\\
      \T(\theta z)=d_{4}\\
      \T(\theta \beta z)=d'_{4}
    \end{cases}
  \end{align}
  cannot have a solution. We express that this systems of equations cannot have a solution, using Lemma \ref{tracesysteemalg}. We find that $\eta,\theta\in\{\beta,\beta+1,\frac{1}{\beta},\frac{1}{\beta}+1,\frac{1}{\beta+1},\frac{\beta}{\beta+1}\}$ and for each of these choices for $\eta$ or $\theta$ there is a corresponding relation on $d_{3}$ or $d'_{3}$, or on $d_{4}$ or $d'_{4}$, respectively.
  \par \textbf{Claim:} we claim that two different $(h-2)$-spaces in $H$ meeting $\rho$ in $\tau$ cannot give rise to the same linear set in $\B(H)$ unless $\beta^{2}+\beta+1=0$. We look at two arbitrary $(h-2)$-spaces $\nu_{1}$ and $\nu_{2}$ in $H$ meeting $\rho$ in $\tau$, whose points outside $\tau$ are given by $\{(x,\zeta_{1},0)_{\F_2}\mid x\in \F_q,\T(x)=e_{1},\T(\beta x)=e'_{1}\}$ and $\{(x,\zeta_{2},0)_{\F_2}\mid x\in \F_q,\T(x)=e_{2},\T(\beta x)=e'_{2}\}$, respectively. The sets $\B(\nu_{1})\setminus\{N\}$ and $\B(\nu_{2})\setminus\{N\}$ are given by $\{(x,1,0)\mid x\in \F_q,\T(\zeta_{1}x)=e_{1},\T(\zeta_{1}\beta x)=e'_{1}\}$ and $\{(x,1,0)\mid x\in \F_q,\T(\zeta_{2}x)=e_{2},\T(\zeta_{2}\beta x)=e'_{2}\}$, respectively. By Lemma \ref{tracesysteemalg} the sets $\B(\nu_{1})$ and $\B(\nu_{2})$ are equal if and only if either $(\zeta_{1},e_{1},e'_{1})=(\zeta_{2},e_{2},e'_{2})$ or else $\beta^{2}+\beta+1=0$ and $(\zeta_{1},e_{1},e'_{1})\in\{(\beta\zeta_{2},e'_{2},e_{2}+e'_{2}),((\beta+1)\zeta_{2},e_{2}+e'_{2},e_{2})\}$. In the first case we find $\nu_{1}=\nu_{2}$. We conclude that if different $\nu_{1}$ and $\nu_{2}$ determine the same linear set on $\ell$ then $\beta^{2}+\beta+1=0$. We note that such an element $\beta$ only can exist if $\F_{4}$ is a subfield of $\F_{q}$, equivalently if $h$ is even. This finishes the proof of the claim.
  \par  We consider the $(h-2)$-space $\mu_{34}=\left\langle\sigma_{3},\sigma_{4}\right\rangle\cap H$, which is given by $\tau\cup\{(x,\eta+\theta,0)_{\F_2}\mid x\in \F_q,\T(x)=d_{3}+d_{4},\T(\beta x)=d'_{3}+d'_{4}\}$. Now, we distinguish between two cases. First, we assume that $\mu_{12}\neq\mu_{34}$. On the one hand, since $\B(\mu_{12})=\B(\mu_{34})$, we must have by our claim that $\beta^{2}+\beta+1=0$ and $\eta+\theta\in\{\beta,\beta+1\}$ as $\mu_{12}\setminus\tau$ is given by $\{(x,1,0)_{\F_2}\mid x\in \F_q,\T(x)=1,\T(\beta x)=b\}$. On the other hand $\eta,\theta\in\{\beta,\beta+1,\frac{1}{\beta},\frac{1}{\beta}+1,\frac{1}{\beta+1},\frac{\beta}{\beta+1}\}=\{\beta,\beta+1\}$ as solutions of the system of equations in \eqref{eq1}. These two results contradict each other, since it follows from $\eta,\theta\in\{\beta,\beta+1\}$ that $\eta+\theta\in\{0,1\}$.
  \par So, secondly we assume that %$\beta^{2}+\beta+1\neq0$. We consider the $(h-2)$-space $\mu_{34}=\left\langle\sigma_{3},\sigma_{4}\right\rangle\cap H$, which is given by $\tau\cup\{(x,\eta+\theta,0)_{2}\mid\T(x)=d_{3}+d_{4},\T(\beta x)=d'_{3}+d'_{4}\}$. Since $\mu_{12}$ and $\mu_{34}$ determine the same linear set on $\ell$, we know that 
  $\mu_{12}=\mu_{34}$ and hence $\theta=\eta+1$ and $(d_{4},d'_{4})=(d_{3}+1,d'_{3}+b)$. %Now we express that the systems of equations in \eqref{eq1} cannot have a solution, using Lemma \ref{tracesysteemalg}.
  We expressed before that the systems of equations in \eqref{eq1} do not have a solution, and so we find that $\{\eta,\theta\}$ equals $\{\beta,\beta+1\}$, $\{\frac{1}{\beta},\frac{1}{\beta}+1\}$ or $\{\frac{1}{\beta+1},\frac{\beta}{\beta+1}\}$, and for each of these solutions there is a corresponding relation on $d_{3}$ or $d'_{3}$. As $\eta$ en $\theta$ are interchangeable we can choose $\eta\in\{\beta,\frac{1}{\beta},\frac{1}{\beta+1}\}$. We distinguish between these three cases.
  \begin{itemize}
    \item If $\eta=\beta$, then expressing that the system of equations \eqref{eq1} does not have a solution and using Lemma \ref{tracesysteemalg}, it follows that $b=d_{3}+1$. We define $a=d'_{3}+1$ and we find, using $(d_4,d_4')=(d_3+1,d_3'+b)$
    \begin{align}\label{eq2}
      &\A_{1}=\{(x,0,1)\mid x\in \F_q, \T(x)=0,\T(\beta x)=0\}\nonumber\\
      &\A_{2}=\{(x,1,1)\mid x\in \F_q, \T(x)=1,\T(\beta x)=b\}\nonumber\\
      &\A_{3}=\{(x,\beta,1)\mid x\in \F_q,\T(x)=b+1,\T(\beta x)=a+1\}\nonumber\\
      &\A_{4}=\{(x,\beta+1,1)\mid x\in \F_q,\T(x)=b,\T(\beta x)=a+b+1\}\;.
    \end{align}
    \item If $\eta=\frac{1}{\beta}$, then from Lemma \ref{tracesysteemalg} it follows that $d'_{3}=0$. Again using $(d_4,d_4')=(d_3+1,d_3'+b)$, we find
    \begin{align*}
      &\A_{1}=\{(x,0,1)\mid x\in \F_q,\T(x)=0,\T(\beta x)=0\}\\
      &\A_{2}=\{(x,1,1)\mid x\in \F_q,\T(x)=1,\T(\beta x)=b\}\\
      &\A_{3}=\{(x,\frac{1}{\beta},1)\mid x\in \F_q,\T(x)=d_{3},\T(\beta x)=0\}\\
      &\A_{4}=\{(x,\frac{1}{\beta}+1,1)\mid x\in \F_q,\T(x)=d_{3}+1,\T(\beta x)=b\}\;.
    \end{align*}
    Using the transformation matrices $\left(\begin{smallmatrix}\kappa_{1}&0&0\\0&\kappa_{2}&0\\0&0&1\end{smallmatrix}\right)$ with $(\kappa_{1},\kappa_{2})=(1,\beta)$ if $d_{3}=1$, with $(\kappa_{1},\kappa_{2})=(\beta,1)$ if $b=1$, and with $(\kappa_{1},\kappa_{2})=(\beta+1,\frac{\beta}{\beta+1})$ if $b=d_{3}$, and setting $(\beta',a',b')=(\beta,b+1,0)$, $(\beta',a',b')=(\frac{1}{\beta},d_{3}+1,1)$ and $(\beta',a',b')=(\frac{1}{\beta+1},b+1,b+1)$ respectively, we find that the sets $\A_{i}$, $i=1,\dots,4$, are $\PGL$-equivalent to
    \begin{align*}
      &\A'_{1}=\{(x,0,1)\mid x\in \F_q,\T(x)=0,\T(\beta' x)=0\}\\
      &\A'_{2}=\{(x,1,1)\mid x \in \F_q,\T(x)=1,\T(\beta' x)=b'\}\\
      &\A'_{3}=\{(x,\beta',1)\mid x\in \F_q,\T(x)=b'+1,\T(\beta' x)=a'+1\}\\
      &\A'_{4}=\{(x,\beta'+1,1)\mid x\in \F_q,\T(x)=b',\T(\beta' x)=a'+b'+1\}\;.
    \end{align*}
    We find the same sets as in \eqref{eq2} with the parameters $(a,b)$ replaced by $(a',b')$.
    %These sets are clearly $\PGL$-equivalent to the ones given in \eqref{eq2}.
    \item If $\eta=\frac{1}{\beta+1}$, then from Lemma \ref{tracesysteemalg} it follows that $d'_{3}=d_{3}$. Again using $(d_4,d_4')=(d_3+1,d_3'+b)$, we find
    \begin{align*}
      &\A_{1}=\{(x,0,1)\mid x\in \F_q,\T(x)=0,\T(\beta x)=0\}\\
      &\A_{2}=\{(x,1,1)\mid x\in \F_q,\T(x)=1,\T(\beta x)=b\}\\
      &\A_{3}=\{(x,\frac{1}{\beta+1},1)\mid x\in \F_q,\T(x)=d_{3},\T(\beta x)=d_{3}\}\\
      &\A_{4}=\{(x,\frac{\beta}{\beta+1},1)\mid x\in \F_q,\T(x)=d_{3}+1,\T(\beta x)=d_{3}+b\}\;.
    \end{align*}
    Using the transformation matrices $\left(\begin{smallmatrix}\kappa_{1}&0&0\\0&\kappa_{2}&0\\0&0&1\end{smallmatrix}\right)$ with $(\kappa_{1},\kappa_{2})=(1,\beta+1)$ if $d_{3}=1$, with $(\kappa_{1},\kappa_{2})=(\beta,1+\frac{1}{\beta})$ if $d_{3}=b+1$, and with $(\kappa_{1},\kappa_{2})=(\beta+1,1)$ if $b=0$, and setting $(\beta',a',b')=(\beta,b,1)$, $(\beta',a',b')=(\frac{1}{\beta},b,b)$ and $(\beta',a',b')=(\frac{1}{\beta+1},d_{3}+1,1)$ respectively, we find that these sets are $\PGL$-equivalent to
    \begin{align*}
      &\A'_{1}=\{(x,0,1)\mid x\in \F_q,\T(x)=0,\T(\beta' x)=0\}\\
      &\A'_{2}=\{(x,1,1)\mid x\in \F_q,\T(x)=1,\T(\beta' x)=b'\}\\
      &\A'_{3}=\{(x,\beta',1)\mid x\in \F_q,\T(x)=b'+1,\T(\beta' x)=a'+1\}\\
      &\A'_{4}=\{(x,\beta'+1,1)\mid x\in \F_q,\T(x)=b',\T(\beta' x)=a'+b'+1\}\;.
    \end{align*}
    We find the same sets as in \eqref{eq2} with the parameters $(a,b)$ replaced by $(a',b')$.
    %These sets are clearly $\PGL$-equivalent to the ones given in \eqref{eq2}.
  \end{itemize}
  So, in all three cases we may assume $\A_{1}$, $\A_{2}$, $\A_{3}$ and $\A_{4}$ are given by descriptions given in \eqref{eq2}. Now, the sets $\D_{ij}$ are given by:
  \begin{align*}
    \D_{12}=\D_{34}&=\{(x,1,0)\mid x\in \F_q,\T(x)=1,\T(\beta x)=b\}\\
    \D_{13}=\D_{24}&=\{(x,\beta,0)\mid x\in \F_q,\T(x)=b+1,\T(\beta x)=a+1\}\\&=\{(x,1,0)\mid x\in \F_q,\T(\beta x)=b+1,\T(\beta^{2} x)=a+1\}\\
    \D_{14}=\D_{23}&=\{(x,\beta+1,0)\mid x\in \F_q,\T(x)=b,\T(\beta x)=a+b+1\}\\&=\{(x,1,0)\mid x\in \F_q,\T((\beta+1)x)=b,\T((\beta^{2}+\beta)x)=a+b+1\}
  \end{align*}
  The set $\A\cap\ell$ is the set of points on $\ell\setminus\{N\}$, not in $\D_{12}\cap\D_{13}\cap\D_{14}$. By the previous $\A\cap \ell$ is the set $\{(x,1,0)\mid x\in \F_q,\T(x)=0,\T(\beta^{2} x)=a\}$. So, $\A$ is $\PGL$-equivalent to a KM-arc given in Theorem \ref{ourconstruction}, namely with $\a=\frac{1}{\beta^{2}}$, and consequently $\gamma=\beta$. This finishes the proof in case $\A$ has property (I).
\end{proof}

By combining Lemma \ref{translationprop1}, Lemma \ref{vddtranslation} and Theorem \ref{case2} we obtain the following corollary.
\begin{corollary} A KM-arc $\A$ of type $q/4$ is a translation arc with translation line $\ell$ if and only if $\A$ satifies Property (I) with respect to $\ell$. All translation KM-arcs in $\PG(2,q)$ of type $\frac{q}{4}$, including the example of Vandendriessche, can be obtained from Theorem \ref{ourconstruction}.
  %The example of Vandendriessche can be obtained from Theorem \ref{ourconstruction}. idee: clubs nodig, dan ok.
\end{corollary}

We now exploit the link between $(h-2)$-clubs of rank $h$ in $\PG(1,q)$, $q=2^h$, and translation KM-arcs of type $q/4$ to obtain a classification of $(h-2)$-clubs of rank $h$.

\begin{corollary} Let $\C$ be an $(h-2)$-club of rank $h$ with head $N$ contained in the line $\ell=\PG(1,q)$, $q=2^h$. Then the set $\ell\setminus \C$ together with the head $N$ is an $(h-2)$-club of rank $h-1$. Moreover, $\C$ is equivalent to the set of points $(1,0)\cup \{(x,1)\mid x\in \F_q,\T(x)=1\lor \T(\varepsilon x)=c \}$ for some $\varepsilon \in \F_q$ and $c\in \F_2$.
\end{corollary}
\begin{proof} If $\C$ is an $(h-2)$-club of rank $h$ in $\PG(1,q)$, $q=2^h$, then by Theorem \ref{translation}, $\C$ defines a translation KM-arc $\A$ of type $q/4$ with translation line $\ell$. By Theorem \ref{case2}, $\A$ is $\PGL$-equivalent to a KM-arc obtained by the construction of Theorem \ref{ourconstruction} with $\a=\frac{1}{\beta^{2}}$. This implies that $\A\cap \ell=\{(x,1)\mid x\in\F_q, \T(x)=0,\T(\beta^2 x)=a\}$ for some $\beta\in \F_{2^h}^\ast$ and $a\in \F_2$, which forms together with the point $N$ with coordinates $(1,0)$ an $(h-2)$-club of rank $h-1$.

The set $\A\cap \ell$ is precisely the complement of the set $\C$. Hence, we can describe $\C$ as the set of points $(1,0)\cup \{(x,1)\mid x\in \F_q, \T(x)=1\lor \T(\varepsilon x)=c \}$, for some $\varepsilon \in \F_{2^h}$ and $c\in \F_2$, where we have used that every element of $\F_{2^h}$ is a square.
\end{proof}

%\begin{remark}
%  Voor welke parameters vinden we oude voorbeelden (KM, GW)?
%\end{remark}

\paragraph*{Acknowledgment:} The authors would like to John Sheekey for his help with the computer searches of $i$-clubs and for his help in generalising the example of Theorem \ref{gwthm} to the one of Theorem \ref{iclubgw}. For most of the computer searches we used the GAP package FinInG \cite{fining}.

\end{document}